\documentclass[10pt,a4paper]{amsart}
\usepackage[utf8]{inputenc}
\usepackage[pagebackref,colorlinks,linkcolor=red,citecolor=blue,urlcolor=blue,hypertexnames=false]{hyperref}
\usepackage{csquotes}
\usepackage{amsmath,amsthm,amssymb,mathtools,amsopn}
\usepackage{bbm}
\usepackage{mathrsfs}
\usepackage{amsrefs}
\usepackage{amscd}
\input xy
\xyoption{all}
\newdir{ >}{{}*!/-5pt/@{>}}

\newcommand{\comment}[1]{} 
	
\usepackage{tikz-cd}

\def\N{\mathbb{N}} 
\def\Z{\mathbb{Z}}

\def\S{\mathbb{S}}

\def\topdf{\texorpdfstring}
\theoremstyle{plain}
\newtheorem{thm}[equation]{Theorem}
\newtheorem{lem}[equation]{Lemma}
\newtheorem{coro}[equation]{Corollary} 
\newtheorem{prop}[equation]{Proposition}

\theoremstyle{definition}

\newtheorem{defi}[equation]{Definition} 

\newtheorem{ex}[equation]{Example}
\newtheorem{exas}[equation]{Examples}
\newtheorem{stan}[equation]{Standing assumption}
\newtheorem{lambada}[equation]{$\lambda$-assumption}

\theoremstyle{remark} 

 \newtheorem{rem}[equation]{Remark}
 
  \numberwithin{equation}{section}

\newcommand{\cE}{\mathcal E}

\newcommand{\sH}{\mathscr H}

\newcommand{\cM}{\mathcal M}

\newcommand{\cO}{\mathcal O}
\newcommand{\cP}{\mathcal P}

\newcommand{\cU}{\mathcal U}
\newcommand{\sU}{\mathscr U}
\newcommand{\cUH}{\sU\sH}
\newcommand{\cV}{\mathcal V}
\newcommand{\sV}{\mathscr V}
\newcommand{\cVH}{\sV\sH}

\def\fA{\mathfrak{A}}

\def\fC{\mathfrak{C}}
\def\fD{\mathfrak{D}}

\def\fS{\mathfrak{S}}
\def\fT{\mathfrak{T}}

\def\fU{\mathfrak{U}}

\def\sets{\mathfrak{Set}}

\def\Mod{\operatorname{Mod}}
\def\map{\operatorname{map}}
\def\ind{\operatorname{Ind}}
\def\pro{\operatorname{Pro}}
\def\indu{\operatorname{ind}}
\def\inv{\operatorname{inv}}
\def\res{\operatorname{res}}
\def\can{\operatorname{can}}
\def\hyp{\operatorname{hyp}}
\def\forg{\operatorname{forg}}

\newcommand{\alg}{\rm Alg}
\newcommand{\ring}{\rm Ring}
\newcommand{\ringa}{\ring^*}
\newcommand{\aha}{{{\rm Alg}_\ell}}
\newcommand{\ahas}{{{\rm Alg}^*_\ell}}
\newcommand{\lra}{\longrightarrow}
\newcommand{\iso}{\overset{\cong}{\lra}}
\newcommand{\weq}{\overset{\sim}{\lra}}

\def\triqui{\vartriangleleft}

\def\supp{\operatorname{supp}}
\def\inc{\operatorname{inc}}
\def\can{\operatorname{can}}

\def\diag{\operatorname{diag}}
\def\Gl{\operatorname{GL}}
\def\GL{\Gl}
\def\ad{\operatorname{ad}}

\def\ev{\operatorname{ev}}
\def\id{\operatorname{id}}
\def\iotap{\iota_{+}}

\newcommand{\coker}{{\rm Coker}}
\renewcommand{\ker}{{\rm Ker}}
\newcommand{\im}{\mathrm{Im}}
\newcommand{\op}{\mathrm{op}}

\DeclareMathOperator*{\colim}{colim}

\def\otimesl{\otimes_\ell}

\DeclareMathOperator{\sd}{\rm{sd}}
\newcommand{\sdi}{\sd^\bullet}
\def\ialgl{Alg^*_\ell}

\begin{document}
\hfuzz=22pt
\vfuzz=22pt
\hbadness=2000
\vbadness=\maxdimen

\author{Guillermo Corti\~nas and Santiago Vega}
\title{Bivariant Hermitian K-theory and Karoubi's fundamental theorem}

\begin{abstract} Let $\ell$ be a commutative ring with involution $*$ containing an element $\lambda$ such that $\lambda+\lambda^*=1$ and let $\ahas$ be the category of $\ell$-algebras equipped with a semilinear involution and involution preserving homomorphisms. 
We construct a triangulated category $kk^h$ and a functor $j^h:\ahas\to kk^h$ that is homotopy invariant, matricially and hermitian stable and excisive and is universal initial with these properties. We prove that a version of Karoubi's fundamental theorem holds in $kk^h$. By the universal property of the latter, this implies that any functor $H:\ahas\to\fT$ with values in a triangulated category which is homotopy invariant, matricially and hermitian stable and excisive satisfies the fundamental theorem. We also prove a bivariant version of Karoubi's $12$-term exact sequence. 
\end{abstract}
\maketitle

\section{Introduction}\label{sec:intro}
Let $\ell$ be a commutative ring with involution $*$. Assume that $\ell$ contains an element $\lambda$ such that
\begin{equation}\label{intro:lambda}
\lambda+\lambda^*=1.
\end{equation}
A \emph{$*$-algebra} over $\ell$ is an algebra $A$ equipped with a semilinear
involution $*:A\to A^{\op}$. Write $\aha$ for the category of $\ell$-algebras
and $\ahas$ for the subcategory of $*$-algebras and involution preserving
homomorphisms. We construct a triangulated category $kk^h$, with the same
objects as $\ahas$, and a functor $j^h:\ahas\to kk^h$ which is the identity on
objects and is homotopy invariant, matricially and hermitian stable and excisive
(the first three terms are defined in Subsection \ref{subsec:stab} and the last 
in Definition \ref{def:ex-ho-theory}),
and is universal initial with these properties (see Proposition \ref{prop:kkuniv}). We write $[-1]$ for the suspension functor, and for $n\in\Z$, consider the \emph{bivariant hermitian $K$-theory} groups 
\[
kk^h_n(A,B):=\hom_{kk^h}(j^h(A),j^h(B)[n]),\,\, kk^h(A,B)=kk_0^h(A,B).
\]
Setting $A=\ell$ we recover a Weibel style \cite{kh}, homotopy invariant version of $K^h$, the $K$-theory of hermitian forms of \cite{karfund}. We prove in Proposition \ref{prop:agree} that
\begin{equation}\label{intro:kkkh}
kk_n^h(\ell,A)=KH^h_n(A)\,\, (n\in\Z). 
\end{equation}
The triangulated category $kk^h$ is related to the bivariant $K$-theory category $kk$ of \cite{ct} by means of a pair of functors
\begin{equation}\label{intro:indres}
\res:kk^h\leftrightarrows kk:\indu
\end{equation}
which are both left and right adjoint to each other (Proposition \ref{prop:adjunction}). There is a $*$-algebra $\Lambda$ such that for $\Lambda A=\Lambda\otimesl A$ we have $j^h\circ\Lambda= \indu\circ\res\circ j^h$. Usual homotopy $K$-theory is recovered from $KH^h$ via 
\begin{equation}\label{intro:agree}
KH_*(A)=KH^h_*(\Lambda A). 
\end{equation}
Under the isomorphisms \eqref{intro:kkkh} and \eqref{intro:agree}, the unit and counit maps 
\[
\eta_A\in kk^h(A,\Lambda A) \text{ and }\phi_A\in kk^h(\Lambda A, A)
\]
correspond, respectively, to the forgetful and hyperbolic maps. There are $*$-algebras $U,V\in \ahas$ such that for $UA=U\otimesl A$ and $VA=V\otimesl A$ we have natural transformations $\eta:UA\to \Lambda A$ and $\bar{\phi}:j^h(VA)\to j^h(A)$ which fit into distinguished triangles
\begin{equation}\label{intro:phieta}
\xymatrix{
j^h(U A)\ar[r]& \Lambda j^h(\Lambda A)\ar[r]^{\bar{\phi}}& j^h(A)\ar[r]& j^h(UA)[-1]\\
j^h(VA)\ar[r]& j^h(A)\ar[r]^{\eta}& j^h(\Lambda A)\ar[r]&j^h(VA)[-1].}
\end{equation}
In particular, for $\cVH_*(A):=kk^h_*(\ell, VA)$ and $\cUH_*(A):=kk_*^h(\ell,UA)$, we have long exact sequences
\[
\xymatrix{
\cUH_n(A)\ar[r]& KH_n(A)\ar[r]^{\hyp}& KH^h_n(A)\ar[r]& \cUH_{n-1}(A)\\
\cVH_n(A)\ar[r]& KH_n^h(A)\ar[r]^{\forg}& KH_n(A)\ar[r]&\cVH_{n-1}(A).}
\]
An invertible element $u$ in a unital $*$-algebra is \emph{unitary} if $u^*=u^{-1}$. For unitary $\epsilon\in\ell$  let ${}_\epsilon M_2$ be the algebra of $2\times 2$-matrices with coefficients in $\ell$ equipped with the involution
\[
\begin{bmatrix}a&b\\ c& d\end{bmatrix}^*=\begin{bmatrix}d^*&\epsilon b^*\\ \epsilon^*c^*& a^*\end{bmatrix}
\]
For $A\in\ahas$, we write ${}_\epsilon M_2A={}_\epsilon M_2\otimesl A$. 
A main result of this article is Corollary \ref{coro:bivafund}, which establishes the following.
\begin{thm}\label{intro:bivafund}
Assume that $\ell$ contains an element $\lambda$ satisfying \eqref{intro:lambda}. Then there is a natural isomorphism
\[
j^h({}_{\epsilon} M_2 V(A))\cong j^h({}_{-\epsilon} M_2 U(A))[1]\quad (A\in\ahas).
\]
\end{thm}
Applying $kk^h_*(\ell,-)$ to the isomorphism of Theorem \ref{intro:bivafund} and writing ${}_\epsilon \cVH_*$ and ${}_\epsilon \cUH_*$ for $KH_*$ of ${}_{\epsilon} M_2U$ and ${}_{\epsilon} M_2V$, we obtain
\begin{equation}\label{intro:fundkh}
{}_\epsilon \cVH_*(A)={}_{-\epsilon} \cUH_{*+1}(A).
\end{equation}
We remark that because $j^h$ is hermitian stable and we are assuming that $\ell$ has an element as in \eqref{intro:lambda}, we have a natural isomorphism
\[
j^h(A)\cong j^h({}_1M_2A).
\] 
Hence $\cVH_*={}_1\cVH_*$. Next we explain how Theorem \ref{intro:bivafund} relates to Karoubi's fundamental theorem \cite{karfund}. The latter establishes an equivalence
\begin{equation}\label{intro:echteskar}
{}_\epsilon\sV(R)\weq \Omega  {}_{-\epsilon}\sU(R)
\end{equation}
between the fiber of the forgetful map ${}_\epsilon K^h(R)\to K(R)$ from $K$-theory of $\epsilon$-hermitian forms of a unital $*$-algebra $R$, and the loops of the fiber of the hyperbolic map $K(R)\to {}_{-\epsilon} K^h(R)$. By Lemma \ref{lem:kepsilon},
we have 
\[
{}_\epsilon K^h_*(R)={}_1K_*^h({}_\epsilon M_2R).
\]
Hence we may regard \eqref{intro:fundkh} as a version of \ref{intro:echteskar} for homotopy hermitian $K$-theory and Theorem \ref{intro:bivafund} as a bivariant version. Karoubi uses \eqref{intro:echteskar} to establish a $12$ term exact sequence connecting Witt and co-Witt groups with the $\Z/2\Z$ Tate cohomology of $K$-theory under the action of the involution. We also prove a bivariant version of Karoubi's sequence. The algebra $\Lambda$ is equipped with an involutive automorphism $t$. The latter induces an action of $\Z/2\Z$ on $kk^h(A,\Lambda B)$. Let $\bar{\phi}$ and $\eta$ be as in \eqref{intro:phieta}; set
\begin{align*}
    {}_\epsilon kk_n^h(A,B)&=kk_n^h(A,{}_\epsilon M_2B)\\
		{}_{\epsilon}W_n(A,B)  &= \coker( {}_{\epsilon}kk_n^h(A,\Lambda B) \xrightarrow{\bar{\phi}_*} {}_{\epsilon}kk_n^h(A,B) \\
    {}_{\epsilon}W'_n(A,B) &= \ker  ({}_{\epsilon}kk_n^h(A,B)         \xrightarrow{{\eta}_*} {}_{\epsilon}kk_n^h(A,\Lambda B)) \\
    k_n(A,B)  &= \{ x \in kk_n^h(A,\Lambda B) : x = t_*x\}/\{x = y + t_*y\} \\
    k'_n(A,B)  &= \{ x \in kk^h(A,\Lambda B) : x = - t_*x\}/\{x = y - t_*y\}.
\end{align*}

\medskip

\begin{thm}[cf. \cite{karfund}*{Th\'eor\`eme 4.3}]\label{intro:12}
Assume that $\ell$ contains an element $\lambda$ as in \eqref{intro:lambda}. Let $A,B\in\ahas$ and let $n\in\Z$. There is an exact sequence
\[
\begin{tikzcd}[transform shape, nodes={scale=0.8}, column sep={0.4cm}]
   k_n(A,B) \ar[r] 
 & {}_{-\epsilon }W_{n+1}(A,B) \ar[r] 
 & {}_\epsilon W'_{n-1}(A,B) \ar[r]
 & k_n'(A, B) \ar[r]
 & {}_{-\epsilon}W'_n(A, B)  \ar[r]
 & {}_{-\epsilon}W_n(A, B)  \ar[d] \\
   {}_\epsilon W_n(A, B) \ar[u] 
 & {}_\epsilon W_n'(A,B) \ar[l]
 & k'_n(A, B) \ar[l]
 & {}_{-\epsilon}W'_{n-1}(A,B)  \ar[l] 
 & {}_\epsilon W_{n+1}(A,B) \ar[l]
 & k_n(A,B) \ar[l]
\end{tikzcd}
\]
\end{thm}
By Remark \ref{rem:12}, Theorem \ref{intro:12} follows from Theorem \ref{thm:12}. 

\bigskip

The rest of this article is organized as follows. In Section \ref{sec:prelis} we recall some basic facts and fix notations and vocabulary about involutions and (semi-split) $*$-algebra extensions. We recall how an invertible hermitian element $\Phi$ in a unital $*$-algebra $R$ induces a new involution on $R$ and on any $*$-ideal $A\triqui R$, given by the adjoint with respect to the form $(x,y)\mapsto x^*\Phi y$. We write $A^\Phi$
for $A$ equipped with this new involution. For example ${}_\epsilon M_2$ is $M_2$ with the involution associated to the $\epsilon$-hermitian hyperbolic form. We write $M_\pm$ for $M_2$ equipped with the involution associated to the $1$-hermitian element $\diag(1,-1)$. We give the name ``$\lambda$-assumption" (\ref{stan:lambda}) to the hypothesis that $\ell$ contains an element satisfying \eqref{intro:lambda}. Under this assumption, $M_\pm\cong {}_1M_2$. 
We also introduce the concept of semi-split extension as a $*$-algebra extension which admits a splitting in a fixed choice of underlying category $\fA$ among sets (with involution) and $\ell$-modules (with involution), and a forgetful functor $F:\ahas\to\fA$. We call a $*$-algebra extension
\[
\xymatrix{A\ar@{ >-}[r] &B\ar@{>>}[r]^p& C}
\] 
semi-split if $F(p)$ admits a right inverse in $\fA$. 
In Subsection \ref{subsec:stab} we prove several technical lemmas (Lemmas \ref{lem:x=y}, \ref{lem:conju1} and \ref{lem:xley}) concerning stability with respect to a corner inclusion in the $*$-algebra $M_X$ of finitely supported matrices indexed by a set $X$, which we call $M_X$-stability. We also define the concept of hermitian stability for a functor $H:\ahas\to \fC$. We say that $H$ is hermitian stable if whenever $R$ is a unital $*$-algebra and $\Phi,\Psi\in R$ are invertible hermitian elements of the same parity $\epsilon$ then for every $*$-ideal $A\triqui R$, $H$ maps the upper left hand corner inclusion $A^\Phi\to (M_2A)^{\Phi\oplus\Psi}$ to an isomorphism. We show in Proposition \ref{prop:hermistab} that if $\ell$ satisfies the $\lambda$-assumption and $H:\ahas\to\fC$ is a matricially stable functor, then $H\circ M_{\pm}$ is hermitian stable. Section \ref{sec:kh} is concerned with hermitian $K$-theory. We write $K^h$ for the hermitian $K$-theory used by Karoubi in \cite{karfund} and $KV^h$ for the variant due to Karoubi and Villamayor \cite{kv2}; we also introduce homotopy hermitian $K$-theory $KH^h$. We present $KH^h$ both as the homotopy groups of the total spectrum of a simplicial spectrum, and by means of an algebraic construction, as done in \cite{kh} and \cite{friendly} for usual homotopy algebraic $K$-theory, and explain why these two versions are equivalent. Lemma \ref{lem:khprops} establishes the basic properties of $KH^h$, including excision. We also consider the comparison maps $K^h\to KV^h\to KH^h$. In particular we prove in Lemma \ref{lem:khregular} that if $2$ is invertible in $\ell$, and $A$ is a $K$-regular $*$-algebra which satisfies excision in $K$-theory, then $K^h_*(A)\to KH^h_*(A)$ is an isomorphism. In Section \ref{sec:khprod} we use our algebraic description of $KH^h$ to equip it with a cup-product and show that the latter is compatible with that of $K^h$ via the comparison map (Lemma \ref{lem:prodkhh}). Beginning in Section \ref{sec:stabhom} and for the rest of the paper we assume that $\ell$ satisfies the $\lambda$-assumption  \ref{stan:lambda}. This section introduces, for an infinite set $X$ --which will be fixed for the rest of the article-- a $\Z$-linear category $\{\ahas\}=\{\ahas\}_X$ and a functor $\ahas\to \{\ahas\}$ which is homotopy invariant, $M_X$-stable and hermitian stable and is universal intitial for these properties (Lemmas \ref{lem:braceuniv} and \ref{lem:hermstab}). This section also contains a useful technical lemma (Lemma \ref{lem:fixhomo}) which essentially says that a path of not necessarily involution preserving inner endomorphisms between inner $*$-endomorphisms can be modified to preserve involutions upon stabilization. 
Section \ref{sec:biva} introduces bivariant hermitian $K$-theory. For a fixed choice of infinite set $X$ and of forgetful functor $F:\ahas\to \fU$, we introduce a triangulated category $kk^h$ and a functor $j^h:\ahas\to kk^h$ which is universal among functors to a triangulated category which are homotopy invariant, $M_X$-stable, hermitian stable, and satisfy excision for those $*$-algebra extensions which are semi-split with respect to $F$ in the sense explained above (Proposition \ref{prop:kkuniv}). The category $kk^h$ is also equipped with a tensor product of homomorphisms
\begin{equation}\label{intro:tenso}
\otimes: kk^h(A_1,A_2)\times kk^h(B_1,B_2)\to kk^h(A_1\otimesl B_1,A_2\otimesl B_2),
\end{equation}
defined whenever $A_1\otimesl $ and $A_2\otimesl $ (or $\otimesl B_1$ and $\otimesl B_2$) preserve semi-split extensions (Lemma \ref{lem:tenso}). In Section \ref{sec:copro} we consider, for a pair of split $*$-algebra monomorphisms $C\to A$ and $C\to B$, the coproduct
$A\coprod_CB$ and the amalgamated sum $A\oplus_CB$. There is an obvious quotient map $\pi:A\coprod_CB\to A\oplus_CB$ and we show in Proposition \ref{prop:coprod} that $j^h(\pi)$ is an isomorphism. The proof follows the lines of the analogue result for $j:\aha\to kk$ proved in \cite{ct}*{Theorem 7.1.1}, except that there are certain homotopies that need to be adapted to preserve the involutions; for this we use the technical Lemma \ref{lem:fixhomo} mentioned above. Section \ref{sec:agree} uses the results of the previous sections and the argument of \cite{ct}*{Theorem 8.2.1} to establish the isomorphism \eqref{intro:kkkh}. Under this isomorphism, the product defined in Section \ref{sec:khprod} for $KH^h$ is the particular case of \eqref{intro:tenso} when $A_1=B_1=\ell$. The functors in \eqref{intro:indres} are the subject of Section \ref{sec:indres}. We prove in Proposition \ref{prop:adjunction} that they are both right and left adjoint to each other. The $*$-algebras $U$ and $V$ are introduced in Section \ref{sec:uv}. They are $kk^h$-isomorphic to a desuspension of similar algebras $U'\ell$ and $V'\ell$ introduced by Karoubi in \cite{karfund} (see Remark \ref{rem:uukar}). The main result of this section is Lemma \ref{lem:sigmavu-l}, which, using the adjunction established in the previous one, establishes an isomorphism
\begin{equation}\label{intro:uv}
j^h(UV)[-1]\cong j^h(\ell).
\end{equation}
Here $UV=U\otimesl V$. Using \eqref{intro:tenso} we obtain that the functors $U\otimesl- $ and $V\otimesl-:kk^h\to kk^h$ are category equivalences, and that the suspension of each is inverse to the other. Theorem \ref{intro:bivafund} is proved in Section \ref{sec:fund}, as Corollary \ref{coro:bivafund}. The section starts by recalling an equivalent form of Karoubi's theorem that says that the cup-product with a certain element $\theta_0\in K^h_2({}_{-1}M_2(U'^2))$ induces an isomorphism
\[
\theta_0\star-:K^h_*(R)\iso K^h_{*+2}({}_{-1}M_2U'^2R)
\]
for every unital $*$-algebra $R$. We use Karoubi's result together with the results of the previous sections to show that the composite $\theta\in kk^h(\ell, {}_{-1}M_2U^2)$ of the image of $\theta_0$ under the comparison map with the isomorphism $j^h( {}_{-1}M_2U'^2)[2]\cong j^h( {}_{-1}M_2U^2)$ is an isomorphism (Theorem \ref{thm:bivafund}). Applying \eqref{intro:uv}, we use the latter theorem to deduce Theorem \ref{intro:bivafund} in Corollary \ref{coro:bivafund}. The $12$-term exact sequence of Theorem \ref{intro:12} is established in Section \ref{sec:12}, (Remark \ref{rem:12}) as a consequence of Theorem \ref{thm:12}. The latter combines Karoubi's argument for the proof of his $12$-term exact sequence \cite{karfund}*{Th\'eor\`eme 4.3} with properties of the functors $\Lambda$, $U$ and $V$ established in Sections \ref{sec:indres} and \ref{sec:uv}.

\subsubsection*{Acknowledgements} The first named author wishes to thank Max Karoubi, Marco Schlichting and Chuck Weibel for useful e-mail exchanges.

\section{Preliminaries}\label{sec:prelis}
\numberwithin{equation}{subsection}
\subsection{Rings, algebras and involutions}\label{subsec:rai}
Fix a commutative, unital ring $\ell$ and an involutive automorphism 
\[
{\ \ }^*:\ell\to \ell, \quad x\mapsto x^*.
\]
A unitary $\ell$-bimodule $M$ is called \emph{symmetric} if $\lambda x=x\lambda$ holds for every $\lambda\in\ell$ and $x\in M$. By an \emph{algebra} over $\ell$ we understand a symmetric bimodule $A$ together with an associative, $\ell$-linear multiplication $A\otimesl A\to A$.  An \emph{involution} in an algebra $A$ is an involutive ring homomorphism $A\to A^{\op}$  that is semilinear with respect to the $\ell$-module action; $(\lambda a)^*=\lambda^*a^*$ for $\lambda\in\ell$ and $a\in A$. A \emph{$*$-algebra} is an $\ell$-algebra  equipped with an involution; for $\ell=\Z$ we use the term \emph{$*$-ring}. A \emph{$*$-ideal} in a $*$-algebra $A$ is a two-sided ideal closed under both the $\ell$-bimodule structure and the involution. We write $\aha$ for the category of $\ell$-algebras and $\ell$-algebra homomorphisms and  $\ahas$ for the subcategory of $*$-algebras and involution preserving homomorphisms; we set $\ring=\alg_\Z$ and $\ringa=\alg_\Z^*$. If $A,B$ are $*$-algebras then $A\otimesl B$ is again a $*$-algebra with involution $(a\otimes b)^*=a^*\otimes b^*$. If $L$ is a $*$-algebra and $A\in\ahas$, we shall often write $LA$ for $L\otimesl A$. 

\begin{ex}\label{ex:inv}
Let $A$ be a ring; write  $\inv(A)$ for $A\oplus A^{\op}$ equipped with the involution $(a,b)^*=(b,a)$. If $\ell$ is a commutative ring, then $\inv:\aha\to\alg_{\inv(\ell)}^*$, $A\mapsto \inv(A)$ is a category equivalence, with inverse $B\mapsto (1,0)B$. 
\end{ex}
\begin{exas}[Polynomial and matrix rings]\label{exas:invos}
The identity map of any commutative ring is an involution, the \emph{trivial} involution.
We shall regard the polynomial ring $\Z[t]$ as a ring with trivial involution and $\ell[t]$ with the tensor product involution. For $\lambda\in \ell$, write 
\[
\ev_\lambda:\ell[t]\to \ell,\,\,f\mapsto f(\lambda)
\]
for the evaluation map. Put 
\[
P=\ker\ev_0,\,\, \Omega=\ker\ev_0\cap\ker\ev_1.
\]
We equip the ring $\Z[t,t^{-1}]$ of Laurent polynomials with the involution that interchanges $t$ and $t^{-1}$, and $\ell[t,t^{-1}]$ with the tensor product involution.
For any set $X$ we put
\begin{gather*}
\Gamma_X=\{a:X\times X\to \ell: |\im(a)|<\infty \text{ and }\\
(\exists N)(\forall x\in X)\max\{|\supp(a(x,-)|,|\supp(a(-,x))|\}<N\}.
\end{gather*}
Pointwise addition, the convolution or matricial product and transpose involution, given by \((ab)(x,y)=\sum_za(x,z)b(z,y)\) and \(a^*(x,y)=a(y,x)^*\) respectively, together make $\Gamma_X$ into a $*$-algebra.
This structure is inherited by the $*$-ideal $M_X\triqui\Gamma_X$ of finitely supported functions and by the quotient $\Sigma_X=\Gamma_X/M_X$. If $|X|=n<\infty$, then $M_X=\Gamma_X$ is $*$-isomorphic to $M_n$ with the transpose involution. When $|X|=\aleph_0$, $\Gamma_X$ is isomorphic to Karoubi's cone \cite{kv2}. We write $M_\infty$, $\Gamma$ and $\Sigma$ for $M_\N$, $\Gamma_\N$ and $\Sigma_\N$.
If $\ell$ and $X$ are as above and $A\in\ahas$, we write $A[t]$, $PA$, $\Omega A$, $A[t,t^{-1}]$,$\Gamma_XA$, $M_XA$ and $\Sigma_X A$, for the tensor products of $A$ with $\ell[t]$,$P$,$\Omega$,$\ell[t,t^{-1}]$, $\Gamma_X, M_X$ and $\Sigma_X$, equipped with the tensor product involution.
\end{exas}
\begin{ex}[Functions on a simplicial set]\label{ex:ax}
For $A\in\ahas$ and $n\ge 0$, we regard the polynomial algebra in $n$-variables with coefficients in $A$ as a $*$-algebra via the isomorphism $A[t_1,\dots,t_n]=A\otimes\Z[t_1]\otimes\dots\otimes \Z[t_n]$. 
In particular, for $n\ge 0$ $\Z^{\Delta^n}=\Z[t_0,\dots,t_n]/\langle t_0+\dots+t_n-1\rangle$ is a $*$-ring and 
$A^{\Delta^n}=A\otimes \Z^{\Delta^n}$ is a $*$-algebra. Thus we have a simplicial $*$-algebra
\begin{equation}\label{fun:adelta}
A^{\Delta}:\Delta^{op}\to\ahas,\quad [n]\mapsto A^{\Delta^n}.
\end{equation}
Write $\fS$ for the category of simplicial sets, and $\map_\fS$ for $\hom_\fS$.  If $X$ is a simplicial set and $B_\cdot$ is a simplicial $*$-algebra, then $\map_\fS(X,B_\cdot)$ is a $*$-algebra. Following \cite{ct}, if $X\in\fS$ and $A\in\ahas$ we put
\begin{equation}\label{eq:ax}
A^X=\map_\fS(X,A^\Delta).
\end{equation}
Let $(X,x)$ be a pointed simplicial set, and let $x:pt=\Delta^0\to X$ be the inclusion mapping $0\mapsto x$; write $\ev_x$ for the induced $*$-homomorphism $A^X\to A^{pt}=A$. Put
\begin{equation}\label{eq:axp}
A^{(X,x)}=\ker(A^{X}\overset{\ev_x}{\lra} A).     
\end{equation}
\end{ex}

\begin{ex}[Unitalization]\label{ex:unitadd}
The \emph{unitalization} $\tilde{A}=\tilde{A}_\ell$ of a $*$-algebra $A$ is the  $\ell$-module $A\oplus\ell$ with the following multiplication and involution:
\[
(a,\lambda)(b,\mu)=(ab+\lambda b+a\mu,\lambda\mu),\quad (a,\lambda)^*=(a^*,\lambda^*).
\]
\end{ex}

\begin{exas}[Hermitian elements and involutions]\label{exas:hermele}
Let $R$ be a unital
\goodbreak
\noindent $*$-algebra. An invertible element $\epsilon\in R$ is \emph{unitary} if $\epsilon\epsilon^*=1$. Let $\epsilon\in R$ be a central unitary; an element $\Phi\in R$ is \emph{$\epsilon$-hermitian} if $\Phi^*=\epsilon \Phi$. If $\Phi$ is $\epsilon$-hermitian and invertible  then
\begin{equation}\label{map:invophi}
R\to R, \quad a\mapsto a^{\Phi}:=(\Phi)^{-1}a^*\Phi    
\end{equation} 
is an involution. We write $R^{\Phi}$ for $R$ with the involution \eqref{map:invophi}. If $u\in R$ is invertible and $\Psi=u^*\Phi u$, then $\Psi$ is $\epsilon$-hermitian again and the inner automorphism $\ad(u):R^{\Psi}\to R^{\Phi}$, $x\mapsto uxu^{-1}$ is a $*$-isomorphism. If $S$ is another $*$-algebra and $\Psi$ is $\mu$-hermitian, then $\Phi\otimes\Psi$ is $\epsilon\otimes\mu$-hermitian and 
\begin{equation}\label{eq:rosphi}
(R\otimesl S)^{\Phi\otimes\Psi}=R^{\Phi}\otimesl S^{\Psi}.    
\end{equation}
Let $\epsilon\in\ell$ be a central unitary. Consider the following element of $M_2$
\[
h_{\epsilon}=\left[\begin{matrix}0&\epsilon\\ 1& 0\end{matrix}\right].
\]
To abbreviate notation, we write ${}_\epsilon{M}_2$ for $M_2^{h_\epsilon}\ell$. It follows from \eqref{eq:rosphi} that we have isomorphisms 
\begin{equation}\label{eq:1mepsi=1m2}
{}_1M_2 ({}_{\epsilon}M_2)\cong {}_\epsilon M_2M_2
\end{equation}
Let $X$ be a set; put ${}_\epsilon M_{2X}={}_\epsilon M_2 M_X$. Assume that $X$ is infinite. Using \eqref{eq:1mepsi=1m2} and a bijection $\{1,2\}\times X\iso X$, we obtain $*$-isomorphisms
\begin{equation}\label{map:mepsilonpm}
{}_1M_{2X}({}_\epsilon M_2)\cong
{}_1M_{2}({}_\epsilon M_2) M_{X}\cong
{}_\epsilon M_2M_2M_{X}\cong
{}_\epsilon M_2M_X={}_\epsilon M_{2X}.
\end{equation}
\end{exas}

\begin{ex}\label{ex:phinv} 
Let $\ell_0$ be a commutative ring and $R_0$ a unital $\ell_0$-algebra. As in Example \ref{ex:inv}, consider the $\ell=\inv(\ell_0)$-$*$-algebra  $R=\inv(R_0)$. An element $\epsilon\in R$ is a central unitary if and only if $\epsilon=(\epsilon_0,\epsilon_0^{-1})$ for some central invertible element
$\epsilon_0\in R_0$. Any $\epsilon$-hermitian element is of the form $\Phi=(\Phi_0,\epsilon_0\Phi_0)=(\Phi_0,1)(1,\epsilon_0)(\Phi_0,1)^*$. It follows that $R^\Phi\cong R$.
\end{ex}

In the article we will often assume that $\ell$ satisfies the following. 

\begin{lambada}\label{stan:lambda} $\ell$ contains an element $\lambda$ such that
\begin{equation}\label{lambda}
\lambda+\lambda^*=1.
\end{equation}
\end{lambada}

\begin{ex}\label{ex:lambda}
Hypothesis \ref{lambda} is satisfied, for example, if $2$ is invertible in $\ell$, and also if $\ell=\inv(\ell_0)$ for some commutative ring $\ell_0$.
\end{ex}

Consider the following element of $M_2$
\[
h_{\pm}=\left[\begin{matrix}1&0\\ 0& -1\end{matrix}\right].
\]
For brevity, we write $M_{\pm}=M_2^{h_{\pm}}$ and $M_{\pm}^j=M_{\pm}^{{\otimesl}j}$. Let $\epsilon\in \ell$ be a unitary and $\Phi\in R$ an invertible $\epsilon$-hermitian element. If $\lambda$ is as in \eqref{lambda}, then the element
\begin{equation}\label{elu}
u=u_\lambda=\left[\begin{matrix} 1& 1\\ 
\lambda\Phi^*& -\lambda^*\Phi^*\end{matrix}\right]\in M_2R
\end{equation}
satisfies $u^*(h_\epsilon\otimes 1) u=h_{\pm}\otimes \Phi$. Hence 
we have a $*$-isomorphism
\begin{equation}\label{map:aduiso}
\ad(u): M_{\pm}R^{\Phi}\iso {}_{\epsilon}M_2R.
\end{equation} 
In particular, 
\begin{equation}\label{eq:pm=12}
M_{\pm}\cong {}_{1}{M}_2.
\end{equation}
If $A\triqui R$ is a $*$-ideal for the involution $x\mapsto x^*$, then it is also an ideal for the involution \eqref{map:invophi}; we write $A^{\Phi}$ for $A$ equipped with the latter involution. The isomorphism \eqref{map:aduiso} holds with $A$ substituted for $R$. 

\goodbreak


\subsection{Extensions}\label{subsec:ext}
 An \emph{extension} of $*$-algebras is a sequence
\begin{equation}\label{ext}
\xymatrix{A\ar@{ >-}[r]^i& B\ar@{->>}[r]^p& C}
\end{equation}
such that $p$ is surjective and $i$ is a kernel of $p$. The extension \eqref{ext} is \emph{split} if there is a $*$-homomorphism $C\to B$ which is a section of $p$.

\begin{exas}\label{exas:ext}
Let $A\in\ahas$ and $X$ a set. Then 
\begin{gather}\label{ext:poly}
\xymatrix{PA\ar@{ >-}[r]& A[t]\ar@{->>}[r]^{\ev_0}& A}\\ \label{ext:path}
\xymatrix{\Omega A\ar@{ >-}[r]& PA\ar@{->>}[r]^{\ev_1}& A}\\ \label{ext:cone}
\xymatrix{M_XA\ar@{ >-}[r]&\Gamma_XA\ar@{->>}[r]&\Sigma_XA}
\end{gather}
are extensions. Remark that \eqref{ext:poly} is split by the inclusion $A\subset A[t]$, which is a $*$-homomorphism. The map $A\to PA$, $a\mapsto at$ is an involution-preserving, $\ell$-linear splitting of \eqref{ext:path}. By \cite{ct}*{first paragraph of page 92}, \eqref{ext:cone} admits an $\ell$-linear splitting.  
\end{exas}

\subsection{Underlying data and semisplit extensions}\label{subsec:under} A $*$-algebra may be viewed as an $\ell$-module, an $\ell$-module with involution, or a set (with involution). Each of these 4 choices of underlying data leads to an \emph{underlying category} $\fU$ and a forgetul functor $F:\ahas\to\fU$ which admits a left adjoint $T':\fU\to\aha$. In what follows, we shall fix a choice of $\fU$, $F$ and $T'$. We say that a surjective homomorphism $f$ of $*$-algebras is \emph{semi-split} if $F(f)$ admits a right inverse in $\fU$. An extension \eqref{ext} is \emph{semi-split} if $p$ is. If $\ell$ satisfies the $\lambda$-assumption \ref{stan:lambda}, then every map in $\fU$ which is a section of a $*$-algebra homomorphism can be averaged to a involution preserving section. Thus under assumption \ref{stan:lambda}, a linearly split extension is semisplit for any of the aforementioned choices of underlying data.
\par

\subsection{Stability}\label{subsec:stab}
Let $G_1,G_2:\ahas\to\ahas$ be functors and $\iota:G_1\to G_2$ a natural transformation. Let $\fC$ be a category, $H:\ahas\to\fC$ a functor, and $A\in\ahas$. We say that $H$ is \emph{$\iota$-stable on $A$} if the natural map $H(\iota_A):H(G_1(A))\to H(G_2(A))$ is an isomorphism. We say that $H$ is \emph{$\iota$-stable} if it is so on every $A\in\ahas$. 

We say that $H$ is \emph{homotopy invariant} if it is stable with respect to the canonical inclusion $\inc:A\to A[t]$. We say that $A\in\ahas$ is \emph{$H$-regular} if $H$ is $\inc$-stable on $A[t_1,\dots,t_n]$ for every $n\ge 0$. 
Let $X$ be a set; if $x,y\in X$, write $E_{x,y}$ for the matrix unit.
\[
\iota_x:A\to M_X A,\quad a\mapsto E_{x,x}\otimes a.
\]

\begin{lem}\label{lem:x=y}
Let $X$ be a set, $x,y\in X$, $\fC$ a category, and $H:\ahas\to\fC$ a functor. If $H$ is $i_x$-stable on $A$ and on $M_XA$, then $H(i_x)=H(i_y)$. In particular, $H$ is $i_y$-stable on $A$.  
\end{lem}
\begin{proof}
The proof is essentially the same as that of \cite{friendly}*{Lemma 2.2.4}. One shows that there are permutation matrices $\sigma_2$ and $\sigma_3$ in $M_{X\times X}=M_X\otimesl M_X$ of orders $2$ and $3$, respectively, both of which conjugate $(i_x\otimesl M_X)i_x$ into $(i_y\otimesl M_X)i_y$. Observe that permutation matrices are unitary, so they induce $*$-automorphisms.
\end{proof}

In view of Lemma \ref{lem:x=y}, one says that a functor is \emph{$M_X$-stable} to mean that it is $i_x$-stable for some (and then any) element $x\in X$. 

\begin{lem}\label{lem:conju1}
Let $X$ be a set with at least two elements. Let $H:\ahas\to\fC$ be an $M_X$-stable functor, $A\subset B\in\ahas$ and
$u\in B$ such that
\[
uA, Au^*\subset A,\ \ au^*ua'=aa'\ \ (a,a'\in A).
\]
Assume that $H$ is $M_X$-stable on both $A$ and $M_X A$. Then $\ad(u):A\to A$, $\ad(u)(a)=u^*au$ is a $*$-homomorphism, and $H(\ad(u))=\id_{H(A)}$.
\end{lem}
\begin{proof} Immediate from Lemma \ref{lem:x=y} and the argument of \cite{friendly}*{Proposition 2.2.6}
\end{proof}

\begin{lem}\label{lem:xley}
Let $X$ be as in Lemma \ref{lem:conju1} and let $Y$ be a set with $|Y|>|X|$. Then any $M_Y$-stable functor $H:\ahas\to \fC$ is also $M_X$-stable. 
\end{lem}
\begin{proof} We may assume that $X\subset Y$; write $\inc:M_X\subset M_Y$ for the canonical inclusion. Let $x\in X$ and and $\iota=\iota_x$. Because $H$ is $M_Y$-stable, $H(\inc\circ\iota)$ is an isomorphism. Hence $H(\iota)$ is a split monomorphism and $H(\inc)$ is a split epimorphism. Let $\tau:M_X\otimesl M_Y\iso  M_Y\otimesl M_X$, $\tau(a\otimes b)=b\otimes a$. We have $\tau\circ(\iota\otimes M_Y)\circ \inc=\inc\otimes\iota$. Let $\sigma:Y\times X\to Y\times X$ be any bijection which restricts to coordinate permutation on $X\times \{x\}$; also write $\sigma$ for the corresponding permutation matrix in $M_{Y\times X}=M_Y\otimesl M_X$. Then $\ad(\sigma)\circ (\inc\otimes\iota)=\iota\otimes M_X$. Hence
$H(\inc\otimes\iota)$ is an isomorphism, and thus $H(\inc)$ is a split monomorphism. This finishes the proof.
\end{proof}
Let $H:\ahas\to\fC$ be a functor and $A\in\ahas$. We say that $H$ is \emph{hermitian stable} on $A$ if for every $*$-ideal embedding $A\triqui R$ into a unital $*$-algebra $R$, every central unitary $\epsilon\in R$, and any two $\epsilon$-hermitian invertible elements $\Phi,\Psi\in R$, $H$ sends the upper left hand corner inclusion
\[
\iota^\Phi:A^\Phi\to (M_2A)^{\Phi\oplus\Psi}
\]
to an isomorphism.

\begin{prop}\label{prop:hermistab}
Assume that $\ell$ satisfies the $\lambda$-assumption \ref{stan:lambda}. Let $H:\ahas\to\fC$ be a functor and let $A\in\ahas$. If $H$ is $M_2$-stable on ${}_\epsilon M_2 A$ for every $*$-ideal embedding $A\triqui R$ and every central unitary element $\epsilon\in R$, then $H\circ M_{\pm}$ is hermitian stable on $A$.
\end{prop}
\begin{proof}
The isomorphisms $M_\pm A^\Phi\cong {}_\epsilon M_2A$ and
$M_\pm (M_2A)^{\Phi\oplus\Psi}\cong {}_\epsilon M_2M_2A$ of \eqref{map:aduiso} and the permutation isomorphism ${}_\epsilon M_2M_2\cong M_2({}_\epsilon M_2)$ fit into a commutative diagram
\[
\xymatrix{
M_{\pm}A^\Phi\ar[d]^{M_\pm\iota_\Phi}\ar[r]^\cong& {}_\epsilon M_2A\ar[d]^{{}_\epsilon M_2\iota_1}\ar[dr]^{\iota_1}&\\
M_{\pm}(M_2A)^{\Phi\oplus\Psi}\ar[r]_\cong&{}_\epsilon M_2M_2A\ar[r]_{\cong}&M_2({}_\epsilon M_2A)}
\]
The proposition is now immediate.
\end{proof}

\section{Hermitian \topdf{$K$}{K}-theories}\label{sec:kh}
\numberwithin{equation}{section}
Let $A$ be a $*$-ring; write
\[
\cU(A)=\{x\in A\,|\, xx^*=x^*x,\ \ x+x^*+xx^*=0\}.
\]
Observe $\cU(A)$ is a group under the operation
\[
x\cdot y=x+y+xy.
\]
If $A$ is unital then $\cU(A)$ is isomorphic to the group of unitary elements of $A$ via the map $a\mapsto 1+a$. 

Let $A\triqui R$ be a $*$-ideal embedding into a unital $*$-algebra $R$ and let $\epsilon\in R$ be a central unitary. Put
\[
{}_\epsilon \cO(A)=\cU({}_\epsilon M_2 M_\infty A).
\]
By \eqref{map:mepsilonpm} we have a group isomorphism
\begin{equation}\label{map:oepsilono1}
{}_\epsilon \cO(A)\cong {}_1 \cO({}_\epsilon M_2 A).
\end{equation}

The \emph{$\epsilon$-hermitian $K$-theory groups} of a unital $*$-ring $A$ are the stable
homotopy groups of a spectrum ${}_\epsilon K^hA=\{{}_\epsilon K^hA_n\}$ whose $n$-th 
space is  ${}_\epsilon K^hA_n=\Omega B{}_\epsilon \cO(\Sigma^{n+1}A)^+$, the loopspace of 
the $+$-construction (\cite{loday}*{Section 3.1.6}). As usual we also write
\[ 
{}_{\epsilon}{K}_n^h(A) = \pi_n ({}_\epsilon K^hA ),\,\,(n \in \Z)
\]
for the $n$-th stable homotopy group. When $\epsilon = 1$ we drop it from the notation. For a nonunital $*$-ring $A$, we put
\[
{}_{\pm 1}K_n^h(A)=\ker ({}_{\pm 1}K_n^h(\tilde{A}_\Z)\to {}_{\pm}K_h^h(\Z)).
\]
If $A$ is unital, these groups agree with those defined above since hermitian $K$-theory of unital $*$-rings is additive. 
Recall that a ring $A$ is \emph{$K$-excisive} if for any embedding
$A\triqui R$ as an ideal of a unital ring $R$ and every unital homomorphism $R\to S$ mapping $A$ isomorphically onto an ideal of $S$, the map of relative $K$-theory spectra $K(R:A)\to K(S:A)$ is an equivalence.
The definition of $K^h={}_1K^h$-excisive $*$-ring is analogous.
Now assume that $A$ is a $K$-excisive ring that is a $*$-algebra over some commutative ring $\ell$ with involution that satisfies the $\lambda$-assumption \ref{stan:lambda} and let $A\triqui R$ be a $*$-ideal embedding into a unital $*$-algebra. 
Let $f:R\to S$ be a unital $*$-algebra homomorphism mapping 
$A$ isomorphically onto a $*$-ideal of $S$. Let $\epsilon\in \ell$ be a central unitary. It follows straightforwardly from Karoubi's fundamental theorem \cite{karfund}*{Th\'eor\`eme} (see \cite{battikh}*{Corollary 3.5.1}), that the map ${}_\epsilon K^h(R:A)\to {}_\epsilon K^h(S:A)$ is an equivalence. In particular, if $A$ is $K$-excisive then it is also $K^h$-excisive. Taking all this into account, and assuming that $\ell$ satisfies the $\lambda$-assumption \ref{stan:lambda}, 
we set, for any  $K$-excisive $A\in\ahas$, any unitary $\epsilon\in\ell$ and $n\in\Z$, 
\begin{equation}\label{eq:khnnuni}
{}_\epsilon K_n^h(A)=\ker({}_\epsilon K_n^h(\tilde{A})\to {}_\epsilon K_n^h(\ell)).
\end{equation}
For $n\le 0$ and not necessarily $K$-excisive $A$, we take \eqref{eq:khnnuni} as a defintion. The nonpositive hermitian $K$-groups agree with Bass' quadratic $K$-groups \cite{bass} for the maximum form parameter. In particular, by \cite{bass}*{Chapter III, Theorem 1.1}, hermitian $K$-theory as defined above satisfies excision in nonpositive dimensions.  

\begin{rem}[$K^h_0$ and quasi-homomorphisms]\label{rem:qhoms}
For unital $R\in \ringa$, $K_0^h$ is the group completion of the monoid $\cV^h_\infty(R)$ of unitary conjugacy classes of self-adjoint idempotents in ${}_1M_2M_\infty R$. We may thus regard elements of $K_0^h(R)$ as differences of classes of $*$-homomorphisms $\Z\to {}_1M_2 M_\infty R$; this can be formalized in terms of quasi-homomorphisms. Let $A,B,C$ be $*$-rings such that $B \lhd C$. A \emph{quasi-homomorphism} from $A\to B$ through $C$, $(f_+,f_-):A\rightrightarrows C\rhd B$, consists of two $*$-homomorphisms $f_+,f_-: A \to C$ such that  $f_+(a)-f_-(a) \in B$ for all $a\in A$. Thus an element of $K_0^h(R)$ is a class of quasi-homomorphism $\Z\rightrightarrows {}_1M_2M_\infty R\rhd {}_1M_2M_\infty R$. Moreover it follows from \eqref{eq:khnnuni} that if $A$ is any $*$-ring then any element of $K_0^h(A)$ is represented by a quasi-homomorphism $\Z\rightrightarrows {}_1M_2 M_\infty \tilde{A}_\Z\rhd {}_1M_2 M_\infty A$. If $A\in\ahas$ then the same holds with $\ell$ substituted for $\Z$ and $\ell$-linear, $*$-quasi-homomorphisms. 
\end{rem}

For any, not necessarily unital $*$-ring $A$, and $\epsilon=\pm 1$, Karoubi and Villamayor also introduce hermitian $K$-groups for $n\ge 1$. They agree with the homotopy groups of the simplicial group ${}_\epsilon O(A^{\Delta})$ up to a degree shift
\[
{}_\epsilon KV_n^h(A)=\pi_{n-1}{}_\epsilon O(A^\Delta)\quad (n\ge 1).
\]
The argument of \cite{friendly}*{Proposition 10.2.1} shows that the definition above is equivalent to that given in \cite{kv2}; we have
\[ 
{}_\epsilon KV_{n+1}^h(A)={}_\epsilon KV^h_1(\Omega^nA)\qquad (n\ge 1). 
\]
Similarly, if $A$ is unital, for all $n\ge 1$ we have 
\begin{equation}\label{eq:kvpibo}
{}_\epsilon KV^h_n(A)=\pi_nB{}_\epsilon O(A^\Delta)=\pi_nB{}_\epsilon O(A^\Delta)^+=\pi_{n}\Omega B{}_\epsilon O(\Sigma A^\Delta)^+.
\end{equation}

Applying ${}_\epsilon K^h_n$ to the path extension \eqref{ext:path} and using excision, we obtain a natural map
\[
{}_\epsilon K^h_n(A)\to {}_\epsilon K^h_{n-1}(\Omega A) \quad (n\le 0).
\]
For $n\in\Z$, the $n^{th}$ \emph{homotopy $\epsilon$-hermitian $K$-theory} group of $A$ is
\[
{}_\epsilon KH^h_n(A)=\colim_m {}_\epsilon K^h_{-m}(\Omega^{m+n}A).
\]
One can also describe ${}_\epsilon KH^h_n$ in terms of ${}_\epsilon KV^h$; by \cite{kv2}*{Th\'eor\`eme 4.1}, ${}_\epsilon KV^h$ satisfies excision for the cone extension \eqref{ext:cone}. Hence we have a map 
\begin{equation}\label{map:partial}
{}_\epsilon KV^h_n(A)\to {}_\epsilon KV^{h}_{n+1}(\Sigma A).
\end{equation}
The argument of \cite{ct}*{Proposition 8.1.1} shows that
\begin{equation}\label{eq:khykv}
{}_\epsilon KH^h_n(A)=\colim_m {}_\epsilon KV^h_{n+m}(\Sigma^{m}A).
\end{equation}
Now assume that $A$ is unital; let ${}_\epsilon KH(A)$ be the total spectrum of the simplicial spectrum ${}_\epsilon K^h(A^{\Delta})$. We have
\[
\pi_n({}_\epsilon KH^h(A))=\colim_m \pi_{n+m}\Omega B{}_\epsilon O(\Sigma^m A^\Delta)^+=\colim_n{}_\epsilon KV^h_{n+m}(\Sigma^mA)={}_\epsilon KH^h_n(A).
\]

\begin{lem}\label{lem:khprops}
Homotopy hermitian $K$-theory is homotopy invariant, matricially stable and satisfies excision. 
\end{lem}
\begin{proof}
The same argument as in \cite{friendly}*{Theorem 5.1.1} shows this.
\end{proof}

In the proof of the following lemma we use an argument communicated to the first author by Max Karoubi.

\begin{lem}\label{lem:khregular}
Assume that $2$ is invertible in $\ell$. Let $n\in\Z$ and let $A\in\ahas$ be $K_n$-regular. 

\item[i)] If either $n\le 0$ or $A$ is $K$-excisive, then
$A$ is $K^h_n$-regular and $K^h_m(A)\to KH^h_m(A)$ is an isomorphism  for $m\le n$. 
\item[ii)] If $n=0$ then $KV_{m}^hA\to KH^h_{m}A$ is an isomorphism for all $m\ge 1$.
\item[iii)] If $n\ge 1$ and $A$ is $K$-excisive, then $KV_{n+1}^h(A)\to KH_{n+1}^h(A)$ is an isomorphism. 
\end{lem}
\begin{proof}
First assume that $A$ is unital. Consider Karoubi's $12$-term exact sequence \cite{karfund}*{Th\'eor\`eme 4.3}. Because $2$ is invertible in $\ell$, by \cite{mvsnk}*{Corollary 3.3}, $A$ is regular with respect to the functors $k_*$ and $k'_*$ of Karoubi's sequence.
\end{proof}

\begin{lem}\label{lem:kepsilon}
Let $\epsilon\in\ell$ be unitary. If either $n\le 0$ or $A$ is $K^h$-excisive, then there is a canonical isomorphism
\[
{}_\epsilon K_n^h(A)\cong K^h_n({}_\epsilon M_2 A).
\]
Moreover for all $A\in\ahas$ we have canonical isomorphisms 
\goodbreak
${}_\epsilon KV^h_n(A)\cong KV^h_n({}_\epsilon M_2 A)$ $(n\ge 1)$ and ${}_\epsilon KH^h_n(A)\cong KH^h_n({}_\epsilon M_2 A)$ $(n\in \Z)$.
\end{lem}
\begin{proof}
The isomorphism \eqref{map:oepsilono1} comes from \eqref{map:mepsilonpm}, which is canonical up to the choices of an element $\lambda\in\ell$ in the $\lambda$-assumption \ref{stan:lambda} and a bijection $\{1,2\}\times X\to X$. By \cite{loday}*{Lemme 1.2.7}, if $A$ is unital, then varying those choices has no effect on the homotopy type of the induced isomorphism $B({}_\epsilon O(A))^+\cong B ({}_1 O({}_\epsilon M_2 A))^+$. Applying this to $\Sigma^rA^{\Delta ^m}$ $(r\in \Z, m\ge 0)$, we obtain the statement of the lemma for unital $A$. The nonunital case follows from the unital one using split-exactness.
\end{proof}

\section{Homotopy hermitian \topdf{$K$}{K}-theory and cup products}\label{sec:khprod}

Hermitian $K$-theory of unital $*$-rings is equipped with products \cite{loday}*{Chapitre III}. Using that $K^h$ satisfies excision in nonpositive dimensions we obtain, for $R,A\in \ahas$ with $R$ unital, $m\in\Z$ and $n\le 0$, a natural product
\begin{equation}\label{map:prodk0}
{K}_m^h(R) \otimes{K}_n^h(A) \xrightarrow{\star} {K}_{m+n}^h(R\otimesl A).
\end{equation}
If moreover $m\le 0$, we also obtain the product above for not necessarily unital $R$.

\begin{lem}\label{lem:prodindex}
Let $R,S\in\ahas$ be unital and let $I\triqui S$ be a $*$-ideal. Assume that the sequence
\[
0\to R\otimesl I\to R\otimesl S\to R\otimesl (S/I)\to 0
\]
is exact. Let $\partial$ be the index map. Then the following diagram commutes
\[
\xymatrix{K^h_0(R)\otimes K^h_1(S/I)\ar[r]^{\star}\ar[d]^{1\otimes\partial}& K^h_1(R\otimesl(S/I))\ar[d]^{\partial}\\
          K^h_0(R)\otimes K^h_0(I)\ar[r]^{\star}& K_0^h(R\otimesl I).}
\]

\end{lem}
\begin{proof}
Because $R$ is unital, we may regard $K_0^h(R)$ as the group completion of the monoid $\cV^h_\infty(R)$ as in Remark \ref{rem:qhoms}.
If $g,p\in {}_1M_2 M_n(S/I)$ are a unitary and a self-adjoint idempotent, and $\mathbbm{1}_n\in {}_1M_2M_n$ is the identity matrix, then $[g]\star [p]=[g\otimes p+\mathbbm{1}_n\otimes (\mathbbm{1}_n-p)]$. For a lift $h\in \cU({}_1M_2 M_{2n}S)$ of $g\oplus g^{-1}$, the index map sends $[g]$ to $\partial[g]=[h\mathbbm{1}_nh^{-1}]-[\mathbbm{1}_n]$; choosing the lift $h\otimes p+\mathbbm{1}_{2n}\otimes (\mathbbm{1}_{2n}-(p\oplus p))$,
we obtain $\partial([g]\star [p])=\partial[g]\star [p]$.
\end{proof}
\begin{lem}\label{lem:prod-sigmaomega}
Let $m\in\Z$, $n\le 0$ and $R,A\in\ahas$ with $R$ unital. Let $\partial$ be the connecting map associated to the path extension \eqref{ext:path}. Assume that $\max\{n, m+n\}\le 0$. Then the following diagram commutes.
\[
\xymatrix{K_m^h(R)\otimes K_n^h(A)\ar[r]^\star\ar[d]^{1\otimes\partial}& K^h_{m+n}(R\otimesl A)\ar[d]^\partial\\
K_m^h(R)\otimes K_{n-1}^h(\Omega A)\ar[r]^\star& K^h_{m+n-1}(R\otimesl \Omega A)}
\]
\end{lem}
\begin{proof}
Let $\jmath_2:\ell\to \ell\oplus\ell$
\[
\xymatrix{\Omega\ar@{=}[d]\ar[r]&P\ar[d]^{\inc}\ar[r]^{\ev_1}& \ell\ar[d]^{\jmath_2}\\
          \Omega\ar[r]& \ell[t]\ar[r]^{(\ev_0,\ev_1)}& \ell\oplus\ell}
\]
Let $i\le 0$. Applying Lemma \ref{lem:prodindex} with $S=\Sigma[t]$, $I=\Sigma\Omega$ and $R=\Sigma^{-i}\tilde{A}$, and using naturality and excision, we obtain that the boundary map $\partial:K_i^h(A)\to K_{i-1}^h(\Omega A)$ is the cup product with $\partial([1])\in K^h_{-1}(\Omega)$. The proof now follows from associativity of $\star$. 
\end{proof}
\begin{coro}\label{coro:prod-sigmaomega}
Let $R,A\in\ahas$ with $R$ unital and let $m,n\in\Z$. 

\item[i)] There is an associative product
\[
\star:K^h_m(R)\otimes KH^h_n(A)\to KH^h_{m+n}(R\otimesl A). 
\]
\item[ii)] Let $c_*:K^h_*(R)\to KH^h_*(R)$ be the comparison map. Then for all $m\in\Z$ and $\xi\in K^h_m(R)$, $c_m(\xi)=\xi\star c_0([1])$. 
\end{coro}
\begin{proof}
Part i) is immediate from Lemma \ref{lem:prod-sigmaomega} upon taking colimits. For $m\le 0$, part ii) is clear from the construction of $\star$ and the definition of $KH^h$. For $m\ge 1$, $c_m$ factors as the comparison map $c'_m:K^h_m(R)\to KV^h_m(R)$ followed by the iterated boudary map $KV_m^h(R)\to KV_1^h(\Omega^{m-1}R)\to K_0^h(\Omega^m R)$ followed by the canonical map $K_0^h(\Omega^mR)\to KH_0^h(R)$, which is the transfinite composition of $-\star\partial[1] $. So part ii) follows from the fact that $c'(x\star y)=c'(x)\star y$
for all $x\in K^h_*(R)$ and $y\in K_{-1}^h(\Omega)$. 
\end{proof} 
\begin{lem}\label{lem:prodkhh}
Let $A,B \in \ialgl$ and $m,n \in \Z$. Then \eqref{map:prodk0} induces an associative product
\[
\xymatrix{
{KH}_m^h(A) \otimes {KH}_n^h(B)\ar[r]^\star &
{KH}_{m+n}^h(A\otimesl B).
}
\]
If $m\le 0$ or $A$ is unital, then the following diagram commutes
\[
\xymatrix{K^h_m(A)\otimes KH_n^h(B)\ar[r]^\star\ar[d]^{c_m\otimes 1}& KH^h_{m+n}(B)\\
KH^h_m(A)\otimes KH_n^h(B)\ar[ur]^\star&}
\]
\end{lem}
\begin{proof}
As observed above, the boundary map $\partial:K^h_*\to K^h_{*-1}\circ\Omega$ is the cup product with $\partial[1]\in K_{-1}^h(\Omega)$. It follows that the following diagram commutes for all $r\le 0$
\[
\xymatrix{K^h_{r}(A)\otimes K^h_r(B)\ar[d]^{\partial\otimes\partial}\ar[r]^\star&K^h_{2r}(A\otimesl B)\ar[d]^{\partial^2}\\
K^h_{r-1}(\Omega A)\otimes K^h_{r-1}(\Omega B)\ar[r]^\star &K^h_{2r-2}(\Omega^2 A\otimesl B).
}
\]
Taking colimit along the columns we get the desired product map for $r=s=0$. The general case is obtained from the latter applying the suspension and loop functors as many times as appropriate. Commutativity of the diagram of the Lemma follows from Corollary \ref{coro:prod-sigmaomega}. 
\end{proof}

\begin{coro}\label{coro:kh-mod}
Let  $A \in \ialgl$ and $n \in \Z$, then ${}_\epsilon KH_n^h(A)$ is a $KH_0^h(\ell)$-module.
\end{coro}

\section{Stabilization and homotopy}\label{sec:stabhom}
\begin{stan} From now on we shall assume that $\ell$ satisfies the $\lambda$-assumption \ref{stan:lambda}. 
\end{stan}

An \emph{$\ind$-object} in a category $\fC$ is a pair $(C,I)$ consisting of an upward filtered poset $I$ and a functor
$C:I\to \fC$. The $\ind$-objects of $\fC$ form a category $\ind-\fC$ where homomorphisms are given by
\[
\hom_{\ind-\fC}((C,I),(D,J))=\lim_i\colim_j\hom_{\fC}(C_i,D_j).
\]
Any functor $F:\fC\to\fD$ extends to $\ind-\fC\to \ind\fD$ by applying it indexwise; $F(C)_i=F(C_i)$. We shall also use the dual concept of a \emph{pro-object} in a category; $\pro-\fC=(\ind-(\fC)^{\op})^{\op}$
$A^X$; if $X$ is finite, then $A^X=A\otimes \Z^{X}$, by \cite{ct}*{Proposition 3.1.3}.
$\sd\to\id_{\fS}$ (see \cite{GJ}), one obtains a pro-simplicial set $\sdi X=\{\sd^nX\}$, and thus an $\ind$-$*$-algebra 
$A^{\sdi X}$. For $S^1=\Delta^1/\partial\Delta^1$ we write
\begin{equation}
\cP A=A^{\sdi\Delta^1},\quad , A^{\S^1}=A^{\sdi(S^1,pt)},\quad A^{\S^{n+1}}=A^{\S^n}\otimes\Z^{\S^1} \quad(n\ge 0).
\end{equation}
Observe that the two endpoint inclusions $\Delta^0\to \Delta^1$ induce inclusions $\Delta^0\to\sdi(\Delta^1)$ and evaluation maps $\ev_i: A^{\sdi \Delta^1}\to A$, $i=0,1$. Two $*$-algebra homomorphisms $f_0,f_1:A\to B$ are \emph{$*$-homotopic} if there is a homomorphism $H:A\to B^{\sdi \Delta^1}$ in $\ind-\ahas$ such that $\ev_iH=f_i$ for $i=0,1$;
*-homotopy is an equivalence relation, compatible with composition. We write $[A,B]^*$ for the set of homotopy classes of $*$-homomorphisms and $[\ahas]$ for the category with the same objects as $\ahas$ with $\hom_{[\ahas]}(A,B)=[A,B]^*$.
If 
$A,B\in\ind-\ahas$, we write $[A,B]^*=\hom_{\ind-[\ahas]}(A,B)$. 

Let $C\in\aha$ and let $A,B\subset C$ be subalgebras. If $u,v\in C$ satisfy $uAv\subset B$ and $avua'=aa'$ for all $a,a'\in A$, then 
\begin{equation}\label{map:aduv}
\ad(u,v):A\to B,\quad a\mapsto uav    
\end{equation}
is an algebra homomorphism. In this situation, we say that the pair $(u,v)$ \emph{multiplies $A$ into $B$}. Let $(u_i,v_i)\in C^2$ ($i=0,1$) be pairs multiplying $A$ into $B$. A \emph{homotopy} between $(u_0,v_0)$ and $(u_1,v_1)$ is a pair $(u(t),v(t))\in C[t]^2$ which multiplies $A$ into $B[t]$, and such that $(u(i),v(i))=(u_i,v_i)$. In this case $\ad(u(t),v(t)):A\to B[t]$ is a homomorphism and a homotopy between  $\ad(u_0,v_0)$ and $\ad(u_1,v_1)$.  Now suppose that $C$ is a $*$-algebra and that $A,B$ are $*$-subalgebras; if $u\in C$ and $(u,u^*)$ multiplies $A$ into $B$, then $\ad(u,u^*)$ is a $*$-homomorphism, and we say that $u$ \emph{$*$-multiplies} $A$ into $B$. If $v\in C$ is another element with the same property, a \emph{$*$-homotopy} from $u$ to $v$ is an element $w(t)\in C[t]$ $*$-multiplying $A$ to $B[t]$ such that $w(0)=u$ and $w(1)=v$.
We shall often encounter examples of elements $u_0,u_1\in C$ which $*$-multiply  $A$ into $B$ and which are homotopic via a pair $(u,v)$ with $v\neq u^*$ so that the homotopy $\ad(u,v)$ is not a $*$-homomorphism. We shall see in Lemma \ref{lem:fixhomo} below that this can be fixed upon stabilization.

Write $i_+$ and $i_-$ for the upper left and lower right corner inclusions $\ell\to M_{\pm}$; both are $*$-homomorphisms. 

\begin{lem}\label{lem:fixhomo}
Let $C\in\ahas$, let $A,B\subset C$ be $*$-subalgebras and let $u_0,u_1\in C$ $*$-multiply $A$ into $B$. Let $(v,w)\in C[t]^2$ be a homotopy of multipliers from $(u_0,u_0^*)$ to $(u_1,u_1^*)$. Assume that 
\begin{equation}\label{wav}
w^*Aw\subset B[t]\supset vAv^*.    
\end{equation}
Let $\lambda\in\ell$ be as in \eqref{lambda}. Then the element 
\[
c(v,w)=\left[\begin{matrix}\lambda^*v+\lambda w^*&\lambda^*(v-w^*)\\
\lambda(v-w^*)&\lambda v+\lambda^*w^*\end{matrix}\right]
\]
 $*$-multiplies $\iota_+(A)$ into $M_{\pm}B[t]$, and $\ad(c(v,w),c(v,w)^*)\circ\iota_+$ is a $*$-homotopy from $\iota_+\circ\ad(u_0,u_0^*)$ to $\iota_+\circ\ad(u_1,u_1^*)$.
\end{lem}
\begin{proof}
One checks that $c(v,w)^*c(v,w)=c(wv,wv)$. Hence if $a,a'\in A$, then
\begin{align*}
\iota_+(a)c(v,w)^*c(v,w)\iota_+(a')=&\iota_+(a(\lambda^*wv+\lambda (wv)^*)a')\\
&=\iota_+(aa').
\end{align*}
Similarly one checks, using \eqref{wav}, that  $c(v,w)\iota_+(a)c(v,w)^*\subset M_{\pm}B[t]$. Thus $H=\ad(c(v,w))\circ\iota_+:A\to M_{\pm}B[t]$ is a $*$-homomorphism. Moreover for $i=0,1$,
$\ev_i(c(v,w))=c(u_i,u_i)=\diag(u_i,u_i)$. Hence $\ev_i\circ H=\iota_+\ad(u_i,u_i^*)$. 
\end{proof}
Let $p,q\ge 0$ and $n=p+q$. Set
\[
\iota_+^{p,q}:=M_{\pm}^p\otimes\iota_+\otimes M_{\pm}^q:M^n_{\pm}\to M^{n+1}_{\pm}.
\]
\begin{lem}
Let $p,q$ and $n$ be as above, and let $p', q'\ge 0$ be such that $p'+q'=n+1$. Then $\iota_+^{p',q'}\iota_+^{p,q}$ is $*$-homotopic to $\iota_+^{0,n+1}\iota_+^{0,n}$.
\end{lem}
\begin{proof}
We have $\iota_+^{0,0}=\iota_+$ and $\iota_+^{1,0}\iota_+=\iota_+^{0,1}\iota_+$. Next observe that, under the identification $\ell^2\otimesl \ell^2=\ell^{\{1,2\}^2}$,  $\iota_+^{1,0}(e_{i,j})=e_{(i,1),(j,1)}$ and $\iota_+^{0,1}(e_{i,j})=e_{(1,i),(1,j)}$. One checks that the matrix 
\[
u=e_{(1,1),(1,1)}-e_{(1,2),(2,1)}+e_{(2,1),(1,2)}+e_{(2,2),(2,2)}
\]
is a unitary element of $M^2_{\pm}$ and satisfies
$\ad(u)\iota_+^{1,0}=\iota_+^{0,1}$. Moreover by \cite{ct}*{Section 6.4}, there exists an invertible element  $u(t)\in M^2_{\pm}$ such that $u(0)=1$
and $u(1)=u$. Hence the composites of $\iota_+^{0,2}$ with $\iota_+^{1,0}$ and $\iota_+^{0,1}$ are $*$-homotopic by Lemma \ref{lem:fixhomo}. Tensoring on both sides with identity maps, we get that $\iota_+^{p,q+2}\iota_+^{p+1,q}\sim^*\iota_+^{p,q+2}\iota_+^{p,q+1}$. Let $p',q'\ge 0$ such that $p'+q'=n+1$. Permuting factors in the tensor product $M_{\pm}^{n+1}=M_{\pm}^{\otimesl n+1}$ we obtain a $*$-isomorphism $\sigma:M_{\pm}^{n+1}\to M_{\pm}^{n+1}$ such that $\sigma\iota_+^{p,q+2}=\iota_+^{p',q'}$. Hence we have 
\begin{equation}\label{eq:iotas}
\iota_+^{p',q'}\iota_+^{p+1,q}\sim^*\iota_+^{p',q'}\iota_+^{p,q+1}    
\end{equation}
 for all $p,q,p',q'$ as above. The lemma follows from \eqref{eq:iotas} using the identity 
 \[
 \iotap^{r,s+1}\iotap^{r,s}=\iotap^{r+1,s}\iotap^{r,s}
 \]
\end{proof}

Consider the $\ind$-$*$-algebra
\[
M_{\pm}^\bullet=\{\iotap^{0,n}:M^n_{\pm}\to M^{n+1}_{\pm}\}.
\]
Let $X$ be an infinite set. Put 
\[
\cM_X=M^\bullet_{\pm}M_X.
\]
Any bijective map $f:X\to Y$ induces an isomorphism
\[
f_*:\cM_X\iso \cM_Y.
\]
\begin{lem}\label{lem:biyes}
Let $f,g:X\to Y$ be bijections. Then $[f_*]=[g_*]\in [\cM_X, \cM_Y]^*$. 
\end{lem}
\begin{proof}
It suffices to consider the case when $X=Y$ and $g=\id_X$. The matrix $u=\sum_{x\in X}E_{f(x),x}$ is a unitary element of $\Gamma_X$, and $f_*$ is the restriction to $M_X$ of the automorphism $\ad(u):\Gamma_X\to\Gamma_X$. We have $\iota_+\circ\ad(u)=\ad(u\oplus 1)\circ \iota_+$. Using \cite{ct}*{Section 3.4} one obtains a homotopy of multipliers between $\ad(u\oplus 1)$ and $\ad(1\oplus u)$; hence $\iota_+\circ\ad(1\oplus u)\sim^* \iota_+\circ\ad(u\oplus 1)$, by Lemma \ref{lem:fixhomo}. Thus $\iota_+^2\ad(u)\sim^*\iota_+^2$.
\end{proof}

\begin{lem}\label{lem:x=y2}
Let $x,y\in X$. Then $\iota_+\iota_x\sim^*\iota_+\iota_y$.
\end{lem}
\begin{proof}
If $x=y$, the maps of the lemma are equal, hence homotopic. Assume $x\ne y$; let $X'=X\setminus\{x,y\}$. Let $u=E_{y,x}-E_{x,y}+\sum_{z\in X'}E_{z,z}$; $u$ is a unitary element of $\Gamma_X$, and satisfies $\ad(u)\circ\iota_x=\iota_y$. Moreover, by \cite{ct}*{Section 3.4}, there is an invertible element $u(t)\in\Gamma_X[t]$ such that $u(0)=1$ and $u(1)=u$. Hence $\iota_+\iota_x\sim^*\iota_+\iota_y$, by Lemma \ref{lem:fixhomo}.
\end{proof}
In view of Lemma \ref{lem:x=y2}, we shall pick any $x\in X$ and write $\iota_X$, or simply $\iota$, for the homotopy class of $\iota_x$, as well as for that of $\iota_X\otimesl \id_A$ $(A\in\ahas)$.
Because we are assuming that $X$ is infinite, there is a bijection $X\coprod X\iso X$; this induces a $*$-homomorphism  
$\cM_X\oplus\cM_X\to \cM_X$; write $\boxplus$ for its $\ind$-$*$-homotopy class. By Lemma \ref{lem:biyes}, $\boxplus$ is independent of the choice of bijection above. As in \cite{ct}*{Section 4.1} one checks, using the lemmas above, that $\cM_X$ equipped with $\boxplus$ and with the class of the zero map, is an abelian monoid in $[\ind-\ahas]$. Similarly, any choice of bijection $X\times X\to X$ gives rise to the same $\ind$-$*$-homotopy class
$\mu$ of $*$-homomorphism $\cM_X\otimesl\cM_X\to \cM_X$. Again, one checks that $\mu$ is associative with unit the $\ind$-$*$-homotopy class of $\iota:\ell\to \cM_X$ and that it distributes over $\boxplus$. Hence $(\cM_X,\boxplus,\mu, 0,[\iota_+])$ is a semi-ring in $[\ind-\ahas]$. 
Let $A,B\in\ind-\ahas$; put
\begin{equation}\label{hombrace}
\left\{A,B\right\}^*_X=[A,\cM_XB]^*.
\end{equation}
The monoid operation $\boxplus$ on $\cM_X$ induces one on $\left\{A,B\right\}^*_X$. One checks that the product $\mu$ induces a bilinear, associative composition
\begin{gather*}
\star:\left\{B,C\right\}^*_X\times\left\{A,B\right\}^*_X\to \left\{A,C\right\}^*_X\\
[f]\star [g]=[\mu \cM_Xf\circ g]. 
\end{gather*}
Thus we have a category  
$\left\{\ind-\ahas\right\}_X$ with the same objects as $\ind-\ahas$, where homomorphisms are given by \eqref{hombrace} and which is enriched over abelian monoids. 
 Moreover, for $n\ge 1$, $\{A, B^{\S^n}\}^*$ is an abelian group by \cite{ct}*{Theorem 3.3.2} and the Hilton-Eckmann argument. 
 
 There is a canonical functor $\can:[\ahas]\to\{\ahas\}_X$, which is the identity on objects and sends the class of a map $f$ to that of $\iota f$.

 \begin{lem}\label{lem:braceuniv}
The composite functor $\can:\ahas\to [\ahas]\to\{\ahas\}_X$ is homotopy invariant, $M_X$-stable and $\iota_+$-stable. Moreover any functor $H:\ahas\to\fC$ which is homotopy invariant $M_X$-stable and $\iota_+$-stable, factors uniquely through $\can$.
 \end{lem}
 \begin{proof}
Straightforward. 
 \end{proof}
\begin{lem}\label{lem:hermstab}
Let $\epsilon\in\ell$ be a unitary element. Let $R\in\ahas$ be a unital $*$-algebra,  let $\Phi,\Psi\in R$ be $\epsilon$-hermitian elements, and let $A\triqui R$ be a $*$-ideal. Then for every infinte set $X$, the canonical functor $\can:\ahas\to\{\ahas\}_X$ maps the canonical inclusion $A^{\Phi}\to M_2A^{\Phi\oplus\Psi}$ to an isomorphism. 
\end{lem}
\begin{proof} The proof follows from Lemmas \ref{lem:braceuniv} and \ref{lem:xley} and Proposition \ref{prop:hermistab}.
\end{proof}
\begin{ex}\label{ex:cornerid}
Lemma \ref{lem:hermstab} applied to $R=\tilde{A}$ and $\Phi=1$ says that the upper left hand corner inclusion $i:A\to M_2A$ is an isomorphism in $\{\ahas\}_X$.
\end{ex}

\begin{lem}\label{lem:hermstab2}
Let $A\triqui R$ be as in Lemma \ref{lem:hermstab} and let $\lambda_1,\lambda_2\in R$ be central elements satisfying \eqref{lambda}. Let 
\[
p_i=p_{\lambda_i}=\left[\begin{matrix}\lambda_i^*& 1\\  \lambda_i\lambda_i^*& \lambda_i\end{matrix}\right]
\]
and let $\iota_i:A\to {}_1M_2A$, $\iota_i(a)=p_ia$. Then $\can(\iota_1)=\can(\iota_2)$ is an isomorphism in $\{\ahas\}_X$.
\end{lem}
\begin{proof}
Let $u_i=u_{\lambda_i}$ be as in \eqref{elu}. Under the isomorphism \eqref{map:aduiso}, $\iota_i$ corresponds to $\iota_+$. Thus $\can(\iota_i)$ is an isomorphism. Morover, since $u={u_2u_1^{-1}\in {}_1M_2R}$ is unitary, $\can(\ad(u))=\id_{{}_1M_2A}$ by Lemma \ref{lem:conju1}, so  \[\can(\iota_2)=\can(\ad(u_2u_1^{-1}))\can(\iota_1)=\can(\iota_1).\qedhere\]
\end{proof}

\section{Bivariant hermitian \topdf{$K$}{K}-theory}\label{sec:biva}
\numberwithin{equation}{subsection}
\subsection{Extensions and classifying maps}\label{subsec:extclassi}
As in Subsection \ref{subsec:under}, let $\fU$ be our fixed underlying category and $F:\ahas\to\fU$ and $T':\fU\to\ahas$ the forgetful functor and its left adjoint. Put $T=T'F:\ahas\to\ahas$. If $A\in\ahas$, write $JA=\ker(TA\to A)$ for the kernel of the counit of the adjunction. A sequence of $\ind$-$*$ algebras \eqref{ext} is an extension if $p$ is a cokernel of $i$ and $i$ is a kernel of $p$; it is semi-split if $F(p)$ is split. Any semi-split extension \eqref{ext} gives rise to a $*$-homomorphism $JC\to A$, which is unique up to homotopy \cite{ct}*{Proposition 4.4.1}; its homotopy class, and, by abuse of notation, any map in it, is called the \emph{classifying map} of the extension. We write $\gamma_A$ and $\rho_A$ for the classifying maps of the extensions
\begin{gather}
\xymatrix{(JA)^{\S^1}\ar@{ >-}[r] & (TA)^{\S^1}\ar@{->>}[r] & A^{\S^1}}\label{gamma}\\
\xymatrix{A^{\S^1}\ar@{ >-}[r] & \cP A\ar@{->>}[r] & A}\label{rho}    
\end{gather}
For $m,n\ge 0$, put also 
\begin{gather}
\gamma_A^{1,n}=(\gamma_{A})^{\S^{n-1}}\circ\cdots\circ\gamma_{A^{\S^{n-2}}}^{\S^1}\circ\gamma_{A^{\S^{n-1}}}:J(A^{\S^n})\to (JA)^{\S^n},\\
\gamma_A^{m,n}=\gamma^{1,n}_{J^{m-1}A}\circ\cdots\circ J^{m-2}(\gamma_{JA}^{1,n})\circ J^{m-1}(\gamma_A^{1,n}):J^m(A^{\S^n})\to J^m(A)^{\S^n}.
\end{gather}

In the following Lemma and elsewhere we shall abuse notation and use the same letter for the homotopy class of a map $f:A\to B\in\ind\ahas$ and  for its image in $\{A,B\}$, and in case the latter is an abelian group (e.g. if $B=C^{\S^n}$) we put $-f$ for the inverse of $\can(f)$ in that group. 

\begin{lem}\label{lem:correrho}
Let $A\in\ind-\ahas$ and let $X$ be an infinite set. Then the following diagram commutes in $\{\ind-\ahas\}_X$
\[
\xymatrix{J^2(A)\ar[r]^{-\rho_{JA}}\ar[d]_{J(\rho_A)}& JA^{\S^1}\\
          J(A^{\S^1})\ar[ur]_{\gamma_A}&}
\]
\end{lem}
\begin{proof}
The analogue of the lemma in the setting of bornological algebras is proved in \cite{cmr}*{Lemma 6.30}. With some obvious modifications, the same argument works in the present case. 
\end{proof}

\begin{rem}\label{rem:correrho}
The analogue of Lemma \ref{lem:correrho} for algebras without involution also holds as stated, and can be deduced from the lemma using the equivalence of Example \ref{ex:inv}.
\cite{ct}*{Lemma 6.2.2}, where the sign is missing. A sign is also missing in the definition of composition in the category $kk$ \cite{ct}*{Theorem 6.2.3}; this is fixed as in \eqref{compo} below.  
\end{rem}
\comment{
\begin{lem}\label{lem:jfun}
The assignment $A\mapsto J(A)$, $f:A \to B $
\end{lem}
}

\subsection{The category \topdf{$kk^h$}{kkh}}\label{subsec:catkkh}

Fix an infinite set $X$ and let $A,B\in\ahas$. As in \cite{ct}*{Section 6.1}, there is a map $\{A,B\}^*_X\to \{JA,B^{\S^1}\}^*_X$ which sends the class of $f$ to that of $\rho_B J(f)$.
\begin{equation}\label{eq:kk*}
kk^h(A,B)=kk^h(A,B)_X=\colim_n\{J^nA,B^{\S^n}\}^*_X
\end{equation}
Since $X$ is fixed, we drop it from our notation. Define a composition law
\begin{gather}\label{map:compo}
\circ: kk^h(B,C)\otimes kk^h(A,B)\to kk^h(A,C), \\
(\xi,\eta)\mapsto \xi\circ\eta\nonumber
\end{gather}
as follows.
If $\xi$ is represented by a class $[f]\in \{J^mB,C^{\S^m}\}^*$ and $\eta$ is represented by $[g]\in \{J^nA,B^{\S^n}\}^*$, put
\begin{equation}\label{compo}
\xi\circ\eta=[f^{\S^n}\circ (-1)^{mn}\gamma_B^{m,n}]\star [J^m(g)].    
\end{equation}
As in \cite{cmr}*{Section 6.3}, one checks that the composition above is well-defined and associative. Hence we have a category $kk^h$ with the same objects as $\ahas$, where the identity map of $A\in\ahas$ is represented by the class of $\iota_A:A\to \cM_XA$. Define a functor $\{\ahas\}\to kk^h$ as the identity on objects and as the canonical map to the colimit $\{A,B\}_X\to kk^h(A,B)$ on arrows. Composing the latter with the functor $\ahas\to\{\ahas\}_X$ we obtain
\begin{equation}\label{map:j}
j^h:\ahas\to kk^h.
\end{equation}
Let $\fT$ be a triangulated category; write $[-n]$ for the $n$-fold suspension in $\fT$. Let $\cE$ be the class of all semi-split extensions 
\[
(E)\qquad \xymatrix{A\ar@{ >-}[r]^i& B\ar@{->>}[r]^p& C.}
\]
\begin{defi}\label{def:ex-ho-theory}
An \emph{excisive homology theory} on $\ahas$ with values in a triangulated category $\fT$ is a functor $H:\ahas\to\fT$ together with a family of maps
$\{\partial_E:H(C)[1]\to H(A)| E\in\cE\}$ such that for every $E\in\cE$, 
\[
\xymatrix{H(C)[1]\ar[r]^{\partial_E}&H(A)\ar[r]& H(B)\ar[r]& H(C)}
\]
is a triangle in $\fT$, and such that $\partial$ is compatible with maps of extensions (\cite{ct}*{Section 6.6}). 
\end{defi}
Let $f:A\to B\in\ahas$ be a $*$-homomorphism. The \emph{homotopy fiber} of $f$ is the pullback $P_f$ of $f$ and $\ev_1:PB\to B$. Thus the following diagram is cartesian
\begin{equation}\label{diag:fiber}
    \xymatrix{P_f\ar[d]\ar[r]& PB\ar[d]^{\ev_1}\\ A\ar[r]^f& B}
\end{equation}
Observe that $\Omega B$ embeds naturally as a $*$-ideal in $P_f$. The sequence
\begin{equation}\label{seq:hofi}
\Omega B\to P_f\to A\overset{f}\lra B
\end{equation}
is the \emph{homotopy fibration} associated to $f$. 
\begin{prop}\label{prop:kkuniv}
The category $kk^h$ has a triangulation which makes the functor $j^h:\ahas\to kk^h$ into an excisive homology theory which is homotopy invariant, $\iota_+$-stable and $M_X$-stable. Moreover any other excisive, homotopy invariant, $\iota_+$-stable and $M_X$-stable homology theory $H:\ahas\to\fT$ factors uniquely through a triangulated functor $\bar{H}:kk^h\to \fT$.
\end{prop}
\begin{proof} As in \cite{ct}*{Corollary 6.4.2} one proves that the endofunctor of $kk^h$ induced by $\Omega\Sigma_X$ is naturally isomorphic to the identity. For $A\in\ahas$, put $j^h(A)[-n]=j^h(\Sigma^n_XA)$. Thus $j^h(A)\mapsto j^h(A)[-1]$ is an equivalence with inverse $j^h(A)\mapsto j^h(A)[1]:=j^h(\Omega A)$. As in \cite{ct}*{Definition 6.5.1} we call a triangle in $kk^h$ 
\[
j^h(C')[1]\to j^h(A')\to j^h(B')\to j^h(C')
\]
\emph{distinguished} if it is isomorphic to the image under $j^h$ of a homotopy fibration sequence \eqref{seq:hofi}. The proof that the triangles above make $kk^h$ triangulated is the same as in \cite{ct}*{Theorem 6.5.2}. The universal property is proved as in \cite{ct}*{Theorem 6.6.2}.  
\end{proof}

\bigskip

\begin{rem}\label{rem:undinvo}
As explained in Subsection \ref{subsec:under}, the classes of extensions which are semi-split with respect to the underlying categories $\sets$ and $\ell-\Mod$ agree with those semi-split with respect to sets with involution and $\ell$-modules with involution. Hence by Proposition \ref{prop:kkuniv}, the corresponding $kk^h$-theories are the same whether involutions are included in the underlying category or not.
\end{rem}

\begin{rem}\label{rem:flat}
Let $(\ahas)_f\subset\ahas$ and $kk_f^h\subset kk^h$ be the full subcategories whose objects are the $*$-algebras that are flat as $\ell$-modules and let $j_f^h:(\ahas)_f\to kk_f^h$ be the restriction of $j^h$. Observe that $(\ahas)_f$ is closed under $J$ and under homotopy fibers; hence $kk_f^h$ is triangulated and $j_f^h$ is excisive, homotopy invariant, $\iota_+$-stable and $M_X$-stable. Moreover the argument of \cite{ct}*{Theorem 6.6.2} applies to show that $j^h_f$ is universal among such functors.
\end{rem}

\begin{ex}\label{ex:invequi}
Let $\ell_0$ be any commutative ring and let $\ell=\inv(\ell_0)$ and $\inv:\alg_{\ell_0}\to \ahas$ be as in Example \ref{ex:inv}. Then $\inv$ is excisive, homotopy invariant and matricially stable, so it induces a triangulated functor from $kk$ of $\ell_0$-algebras, $kk(\ell_0)$, to 
$kk^h(\ell)$. Similarly, its inverse, $B\mapsto (1,0)B=(\ell^2/(0,1)\ell)\otimes_\ell B$ is excisive, homotopy invariant, and matricially stable; by \ref{ex:phinv} it is also hermitian stable. Hence it induces a functor $kk^h(\ell)\to kk^h(\ell_0)$ which is inverse to $\inv$. This shows that $kk$ is a particular case of $kk^h$.
\end{ex}

\begin{ex}\label{ex:tenso}
Let $L\in\ahas$; then $L\otimesl-$ preserves all $F$-linearly split extensions if either $L$ is flat as $\ell$-module or every $F$-split extension is $\ell$-linearly split. In either case, $j^h(L\otimesl-):\ahas\to kk^h$ is homotopy invariant, matricially stable, hermitian stable and excisive, and therefore induces a triangulated functor $L\otimesl-:kk^h\to kk^h$. By a similar argument, for $kk^h_f$ as in Remark \ref{rem:flat}, any $L\in\ahas$ induces a triangulated functor $L\otimesl-:kk^h_f\to kk^h$.
\end{ex}

\begin{lem}\label{lem:tenso}
Let $A_1,A_2\in\ahas$ such that $A_i\otimesl-$ preserves $F$-linearly split extensions. Then we have a natural bilinear, associative product
\[
kk^h(A_1,A_2)\times kk^h(B_1,B_2)\to kk^h(A_1\otimesl B_1,A_2\otimesl B_2),\,\, (\xi,\eta)\mapsto \xi\otimes \eta
\]
that is compatible with composition in all variables. 
\end{lem}
\begin{proof}
If $F$-split extensions are linearly split, then this is standard \cite{ct}*{Example 6.6.5}. Otherwise $A_i$ is flat for $i=1,2$. By Example \ref{ex:tenso}, $A_i\otimesl-$ and $-\otimesl B_i$
For $\xi\in kk^h(A_1,A_2)$ and $\eta\in kk^h(B_1,B_2)$, set 
\[
\xi\otimes \eta=(\xi\otimes B_2)\circ (A_1\otimes\eta).
\]
It is straightforward to check that the product above has all the desired properties.
\end{proof}

\begin{prop}\label{prop:kkuniv2}
Let $\fC$ be an abelian category and $H:\ahas\to \fC$ a functor. Assume that $H$ is split-exact, homotopy invariant, $\iota_+$-stable and $M_X$-stable. Then there is a unique homological functor 
$\bar{H}:kk^h\to \fC$ such that $\bar{H}\circ j^h=H$.
\end{prop}
\begin{proof} The proof is the same as in \cite{ct}*{Theorem 6.6.6}. 
\end{proof}

\bigskip

Let $\epsilon\in\ell$ be a unitary, $A,B\in\ahas$ and $n\in\Z$. Put
\begin{gather}\label{defi:kk}
kk^h_n(A,B):= \hom_{kk}(j^h(A),j^h(B)[n]), \,\, {}_\epsilon kk^h_n(A,B)=kk^h_n(A,{}_\epsilon M_2 B),\\
kk^h(A,B):=kk^h_0(A,B),\,\, {}_\epsilon kk^h(A,B)={}_\epsilon kk^h_0(A,B).\nonumber
\end{gather}
\begin{rem}\label{rem:epsilon1}
In view of \eqref{map:aduiso}, there is a $*$-isomorphism ${}_1 M_2\cong M_{\pm}$. It follows from this and from Proposition \ref{prop:kkuniv} that for all $A,B\in\ahas$, $\iota_+$ induces a canonical isomorphism
\[
{}_1kk^h_*(A,B)\cong kk^h_*(A,B).
\]
\end{rem}

\begin{ex}\label{exa:mapkkk}
The functor $KH^h_0:\ahas\to KH^h_0(\ell)-\Mod$ satisfies the hypothesis of Proposition \ref{prop:kkuniv2}. Hence the functor $\overline{KH}^h_0$ of the proposition induces a natural homomorphism
\[
kk^h(A,B)\to \hom_{KH^h_0(\ell)}(KH^h_0(A),KH^h_0(B))
\]
Setting $A=\ell$ we obtain a natural map
\begin{equation}\label{map:kktok}
kk^h(\ell,B)\to KH^h_0(B).
\end{equation}
We shall show in Proposition \ref{prop:agree} below that \eqref{map:kktok} is an isomorphism.
\end{ex}

\begin{ex}\label{ex:cones}
Let $f:A\to B$ be a $*$-homomorphism. Then $j^h$ maps \eqref{seq:hofi} to a distinguished triangle, by definition of the latter. One can also fit $f$ into an equivalent triangle by other natural constructions. For example, let $T_f$ be the pullback of $TB\to B$ and $f$; we have a commutative diagram
\begin{equation}\label{diag:Tf=Pf}
\xymatrix{JB\ar[d]\ar[r] &T_f\ar[d]\ar[r]& A\ar@{=}[d]\ar[r]^f& B\ar@{=}[d]\\
          \Omega B\ar[r]& P_f\ar[r]& A\ar[r]_f& B}
\end{equation}
By the argument of \cite{ct}*{Lemma 6.3.10}, the vertical map $JB\to \Omega B$ is a $kk^h$-equivalence. Because the first three terms of the top row above form an extension, it follows that the vertical map $T_f\to P_f$ is a $kk^h$-equivalence. Summing up, the top row is $kk^h$-isomorphic to the bottom row, and is thus a triangle in $kk^h$. For another example, let $\Gamma_f$ be the pullback of $\Sigma f$ and the canonical surjection $\Gamma B\to B$. By the argument of \cite{ct}*{Theorem 6.4.1}, $j^h(\Gamma B)=0$. Hence 
the classifying map  $J\Sigma B\to M_\infty B$ of the cone extension is a $kk^h$-equivalence, and therefore $T_{\Sigma f}\to \Gamma_f$ is a $kk^h$-equivalence. Thus the vertical maps in the commutative diagram below form an isomorphism of triangles
\begin{equation}\label{diag:Tsigmaf=gammaf}
\xymatrix{J\Sigma B\ar[d]\ar[r] &T_{\Sigma f}\ar[d]\ar[r]& \Sigma A\ar@{=}[d]\ar[r]^{\Sigma f}& \Sigma B\ar@{=}[d]\\
          M_\infty B\ar[r]& \Gamma_f\ar[r]& \Sigma A\ar[r]_{\Sigma f}& \Sigma B.}
\end{equation}
It follows that the bottom line of \eqref{diag:Tsigmaf=gammaf} is a distinguished triangle in $kk^h$. The map \eqref{diag:Tsigmaf=gammaf} together with that of \eqref{diag:Tf=Pf} with $\Sigma (f)$ substituted for $f$ is a zig-zag of $kk^h$-equivalences. In particular $j^h(\Gamma_f)\cong j^h(P_{\Sigma f})=j^h(\Sigma P_f)$ and the bottom line of \eqref{diag:Tsigmaf=gammaf} is isomorphic in $kk^h$ to the homotopy fibration associated to $\Sigma f$, which in turn is the suspension of \eqref{seq:hofi}: 
\begin{equation}\label{diag:hofikar=sigmahofi}
\xymatrix{j^h(M_\infty B)\ar[r]\ar[d]^{\cong}&j^h(\Gamma_f)\ar[d]^{\cong}\ar[r]& j^h(\Sigma A)\ar[d]^{\cong}\ar[r]^{j^h(\Sigma f)}&j^h(\Sigma B)\ar[d]^{\cong}
\\
j^h(B)\ar[r]& j^h(P_f)[-1]\ar[r]&j^h(A)[-1]\ar[r]^{j^h(f)[-1]}&\,\, j^h(B)[-1].}
\end{equation} 
\end{ex}
\section{Coproducts}\label{sec:copro}
\numberwithin{equation}{section}
A \emph{retract} of a $*$-algebra $A\in\ahas$ consists of a $*$-algebra $C$ and $*$-homomorphisms $i:C\to A$ and $\alpha:A\to C$ such that $\alpha\circ i=\id_C$. Let $j:C\leftrightarrow B:\beta$ be another retract and consider the coproduct $A\coprod_C B$. Let $I=\ker\alpha$, $J=\ker\beta$. Then $A\coprod_CB$ is the $\ell$-module
\[
C\oplus I\oplus J\oplus I\otimes_{\tilde{C}}J\oplus J\otimes_{\tilde{C}}I\oplus\dots
\]
equipped with the obvious product and involution. The projection onto $C$ is a $*$-homomorphism, and its kernel
is the coproduct $I\coprod J$. The direct sum of all tensors with $2$ factors or more is an ideal in $K\triqui A\coprod_CB$, and the quotient 
\[
A\oplus_CB:=(A\coprod_CB)/K
\]
is the $\ell$-module $C\oplus I\oplus J$ equipped with the summand-wise involution and the product
\[
(c_1,i_1,j_1)(c_2,i_2,j_2)=(c_1c_2, c_1i_2+i_1c_2+i_1i_2, c_1j_2+j_1c_2+j_1j_2).
\]
\begin{prop}\label{prop:coprod}
The projection $\pi: A\coprod_CB\to A\oplus_CB$ is a $kk^h$-equivalence. 
\end{prop}
\begin{proof}
As in \cite{ct}*{Theorem 7.1.1} one reduces to proving the case $C=0$. Let $\diag:A\oplus B\to M_2(A\coprod B)$,
$\diag(a,b)=aE_{1,1}+bE_{2,2}$. Let $i:A\oplus B\to M_2(A\oplus B)$ be as in Example \ref{ex:cornerid}. The matrix $u=E_{1,1}(1,0)+E_{1,2}(0,-1)+E_{2,1}(0,1)+E_{2,2}(1,0)\in M_2\tilde{A}$ is unitary, satisfies
$\ad(u)\circ M_2(\pi)\circ\diag=i$ and is connected to the identity by a path $u(t)\in \GL(\tilde{A}[t])$. Thus  $\iota_+ \circ M_2(\pi)\circ\diag\sim^*\iota_+\circ i$ by Lemma 
\ref{lem:fixhomo}; by Lemma \ref{lem:braceuniv} and Example \ref{ex:cornerid}, this implies that $M_2(\pi)\circ\diag$ is an isomorphism in $\{\ahas\}$ and therefore also in $kk^h$. A similar argument shows that $\diag\circ\pi$ is an isomorphism in $kk^h$; this concludes the proof.
\end{proof}
In the next corollary and elsewhere, for $A\in\ahas$, we write  $QA=A\coprod A$ and $QA\vartriangleright qA=\ker(\id_A\coprod\id_A:QA\to A)$.

\begin{coro}\label{coro:deltiso}
Let $\delta:qA\to A$ be the restriction of the map $\id_A\coprod 0:QA\to A$.
Then $j^h(\delta)$ is an isomorphism in $kk^h$.
\end{coro}
\begin{proof}
This is a formal consequence of Proposition \eqref{prop:coprod}; see \cite{newlook}*{Proposition 3.1 (b)}.
\end{proof}
\begin{rem}\label{rem:sumtria}
Let $\jmath_i:A\to A\oplus A$ and $\pi_i:A\oplus A\to A$ be the two inclusions and the two projections. Omitting $j^h$, we have an isomorphism of triangles in $kk^h$
\[
\xymatrix{qA\ar[d]^{\delta}\ar[r]&QA\ar[r]\ar[d]& A\ar@{=}[d]\\
           A\ar[r]^(.4){\jmath_1-\jmath_2}&A\oplus A\ar[r]^(.6){\pi_1+\pi_2}&A}
\]					
\end{rem}
\section{Recovering \topdf{$KH^h$}{KHh} from \topdf{$kk^h$}{kkh}}\label{sec:agree}

\begin{prop}\label{prop:agree}
The natural map \eqref{map:kktok} is an isomorphism for every $B\in \ahas$.
\end{prop}
\begin{proof}
Write $\beta:kk^h(\ell, B)\to KH^h_0(B)$ for \eqref{map:kktok}. Our proof proceeds much in the same way as \cite{ct}*{Theorem 8.2.1}. By Remark \ref{rem:qhoms}, the set 
$qq(\ell,B)$ of $*$-quasi-homomorphisms
\end{proof}

\begin{coro}\label{coro:agree}
For every unitary element $\epsilon\in\ell$, every $A\in\ahas$ and every $n\in\Z$, there is a natural isomorphism of $KH^h_0(\ell)$-modules
\[
{}_\epsilon kk^h_n(\ell,A) \cong {}_{\epsilon}KH^h_n(A).
\]
\end{coro}

\begin{lem}\label{lem:prodagree}
If either $A$ or $B$ are flat or if $F$-split extensions are $\ell$-linearly split, then under the isomorphism of Proposition \ref{prop:agree} the cup-product of Lemma \ref{lem:prodkhh} corresponds to the tensor product of Lemma \ref{lem:tenso}.  
\end{lem}
\begin{proof} Straightfoward. 
\end{proof}

\section{The functors \topdf{$\indu$}{ind}, \topdf{$\res$}{res} and \topdf{$\Lambda$}{Lambda}}\label{sec:indres}
Let $\inv(\ell)$ be as in Example \ref{ex:inv}; the map $\ell\to \inv(\ell)$, $a\mapsto (a,a^*)$ is a homomorphism of $*$-rings. Composing $\inv:\aha\to\alg^*_{\inv(\ell)}$ with restriction of scalars along the map we have just defined, we obtain a functor $\indu:\aha\to \ahas$.  
The forgetful functor $\res:\ahas\to\aha$ is left adjoint to $\indu$. The unit and counit of this adjunction are the maps
\begin{equation}\label{map:conunit}
\eta_A:A\to A\oplus A^{op},\quad \eta(a)=(a,a^*), \quad pr_1:B\oplus B^{\op}\to B, \quad pr_1(x,y)=x.
\end{equation}
Fix the underlying category $\fA$ to be one of $\sets$ or $\ell-\Mod$ and let $j:\aha\to kk$ and $j^h:\ahas\to kk^h$ be the universal functors. Then $\res$ and $\indu$ induce functors $kk^h\leftrightarrow kk$ which by abuse of notation we shall still call $\res$ and $\indu$. 

\begin{prop}\label{prop:adjunction}
The functors $\res:kk^h\leftrightarrow kk:\indu$ are both right and left adjoint to one another; in other words, for every $A\in\ahas$ and $B\in \aha$ there are natural isomorphisms
\[
kk(\res(A),B)\cong kk^h(A,\indu(B))\text{ and }kk^h(\indu(B),A)\cong kk(B,\res(A)).
\]
\end{prop}
\begin{proof}
The maps \eqref{map:conunit} satisfy the unit and conunit conditions at the algebra level, and therefore also at the $kk$-level; this proves the first adjunction. To prove the second, observe that $\res:\ahas\to\aha$ has a left adjoint $\indu':\aha\to\ahas$, which maps an algebra $B$ to $\indu'(B)=B\coprod B^{\op}$ equipped with the the involution that interchanges the summands. There is natural projection map $\pi_B:\indu'(B)\to \indu(B)$; the same argument as in the proof of Proposition \ref{prop:coprod} shows that $j^h(\pi_B)$ is an isomorphism. This completes the proof.
\end{proof}

\begin{rem} By the proof of Proposition \ref{prop:adjunction}, the unit and counit maps of the second adjunction of the proposition are obtained from those of the adjunction between $\indu'$ and $\res$ using the projection $\pi:\indu'\to \indu$ and the map $\diag:\indu\to M_2\indu'$ of the proof of Proposition \ref{prop:coprod}. Explicitly, for $B\in\aha$ and $A\in \ahas$, let  
\begin{gather}\label{map:psib}
\psi_B:B\to B\oplus B^{\op},\,\,b\mapsto (b,0),\\
\phi_A:A\oplus A^{\op}\to {}_1M_2A,\quad \phi_A(x,y)=\diag(x,y^*).\nonumber
\end{gather}
The unit map is $j(\psi_B)$ and the counit is the composite of $j^h(\phi_A)$ with the image in $kk^h$ of the isomorphism $\can({}_1M_2A)\cong \can(A)$ of Lemma \ref{lem:hermstab2}. 
\end{rem}

Let $\Lambda=\ell\oplus\ell$ equipped with involution
\[
(\lambda,\mu)^*=(\mu^*,\lambda^*).
\]
If $A\in\ahas$ we shall identify $\indu(\res(A))$ with $\Lambda A=\Lambda\otimesl A$ via the isomorphism
\[
\Lambda A\to \indu (\res(A)),\quad (x,y)\mapsto (x,y^*).
\]
Under this identification, the maps $\eta_A$ of \eqref{map:conunit} and $\phi_A$ of \eqref{map:psib} become the scalar extensions of the diagonal embeddings
\begin{gather}
\eta:\ell\to\Lambda,\quad \eta(x)=(x,x)\label{eta},\\
\phi:\Lambda\to {}_1M_2 \quad \phi(x,y)=E_{1,1}x+E_{2,2}y\label{phi}.
\end{gather}

\begin{rem}\label{rem:lambdautoadjoint}
The functor induced by tensoring with $\Lambda$ is left and right adjoint to itself since $\Lambda \cong \indu \res \ell$.
Also, Proposition \ref{prop:adjunction} shows that $kk^h(\cdot, \Lambda(\cdot))=kk(\res(\cdot),\res(\cdot))$. In other words, $\Lambda$ represents $kk$. Note in particular, that
\[
{}_\epsilon kk^h(\cdot,\Lambda(\cdot))=kk^h(\cdot,\Lambda(\cdot))
\]
for any unitary $\epsilon\in \ell$. Moreover by Example \ref{ex:phinv}, if $R\in\ahas$ is unital, $\epsilon\in R$
a central unitary and $\Psi\in R$ an invertible $\epsilon$-hermitian element, then 
\begin{equation}\label{map:lambda-epsilon}
\ad{(1,\Psi^{-1})}:\Lambda R \to \Lambda R^\psi 
\end{equation}
is an isomorphism in $\ahas$. In particular, we have $*$-isomorphisms
\begin{equation}\label{map:lambda-matrix}
\Lambda M_\pm \cong \Lambda ({}_\epsilon M_2) \cong \Lambda M_2.
\end{equation}
 Let 
\begin{equation}\label{eq:twist}
t:\Lambda \to \Lambda,\,\,
    t(x,y) = (y,x).
\end{equation}
Then $t$ is a $*$-homomorphism, $t^2 = \id_\Lambda$, and one checks that the following diagram commutes
\[
\xymatrix{\Lambda\ar[drr]_(.4){E_{1,1}\id+E_{2,2}t\,\,\,}\ar[rr]^{({}_1M_2\eta)\phi}&& {}_1M_2\Lambda\ar[d]^{\cong}\\
&& M_2\Lambda}
\]
Hence
\begin{equation}\label{eq:1+t}
j^h(\iota_1)^{-1}j^h(\eta)j^h(\phi)=\id_{j^h(\Lambda)}+j^h(t).
\end{equation}
\end{rem}

\section{ The functors \texorpdfstring{$U$}{U} and \texorpdfstring{$V$}{V}}\label{sec:uv}

In \ref{diag:fiber} we have defined the homotopy fiber $P_f$ of any $*$-homomorphism $f$. Now consider the homotopy fibers of the maps \eqref{eta} and \eqref{phi}, 
\[
V=P_{\eta},\quad U=P_\phi.
\]
For $A\in\ahas$, put $UA = U\otimesl A$ and $VA = V \otimesl A$; these are, respectively, the homotopy fibers of $\phi\otimes id_A:\Lambda A \to {}_1M_2A$ and $\eta\otimes id_A: A \to \Lambda A$. Because $U$ and $V$ are flat $\ell$-modules, they define functors $U,V:kk^h \to kk^h$, by Example \ref{ex:tenso}.

\begin{rem}\label{rem:uukar}
In Example \ref{ex:cones} we have defined, for every $*$-homomorphism $f$, a $*$-algebra $\Gamma_f$, and shown that $j^h(\Gamma_f)=j^h(P_f)[-1]$. In his paper \cite{karfund}, Karoubi uses, in the case $\ell=\Z$, the letters $U$ and $V$ for $U':=\Gamma_\phi$ and $V':=\Gamma_\eta$. Thus $j^h(U)[-1]=j^h(U'\ell)$ and $j^h(V)[-1]=j^h(V'\ell)$.  Karoubi shows that for every $*$-ring $R$, there are weak equivalences $\Omega K(\Lambda U'R)\weq K^h(\Lambda R)$, $K^h(V'\Lambda R)\weq K^h(\Lambda R)$ and $\Omega K^h(U'V'R)\weq K^h(R)$ (see \cite{karfund}*{Sections 1.3 and 1.4}). Lemmas \ref{lem:ul-vlo} and \ref{lem:sigmavu-l} recast the latter equivalences into the framework of $kk^h$.  
\end{rem}    

\begin{lem}\label{lem:ul-vlo}
There are isomorphisms
\begin{align*}
    j^h(U\Lambda) &\cong j^h(\Lambda) \text{ and}\\
    j^h(V\Lambda) &\cong j^h(\Omega\Lambda).
\end{align*}
\end{lem}
\begin{proof} 
Let us prove the first equivalence. To ease the notation we ommit the functor $j^h$.
Let
\[ U\Lambda \to \Lambda^2 \xrightarrow{\phi \otimes \id_\Lambda} {}_1M_2\Lambda \]
be the triangle in $kk^h$ defining $U\Lambda$. We have an isomorphism 
\begin{align}\label{map:tau}
 \tau:\Lambda^2 &\iso\Lambda \oplus \Lambda\\
 (x_1,x_2) \otimes (x_3,x_4) &\mapsto (x_1 x_3,x_2 x_4, x_1 x_4,x_2 x_3).\nonumber
\end{align}
Put $\lambda_1= (0,1)$, $\lambda_2=(1,0)$  and $\iota_i:\Lambda \to {}_1M_2\Lambda$ as in Lemma \ref{lem:hermstab2}. Let $\jmath_i$ $(i=1,2)$ be as in Remark \ref{rem:sumtria}. Observe that 
\[
((\phi\otimes \id_{\Lambda})\circ\tau^{-1}\circ \jmath_1)(x,y)=\begin{bmatrix} (x,0)& (0,0)\\
(0,0)& (0,y) 
\end{bmatrix}.
\]
The matrix 
\begin{equation}\label{conj-incl}
u=\left[\begin{matrix}
(1,-1)& (1,1)\\
(0,0)& (-1,1)
\end{matrix}\right]\in {}_1M_2(\indu(\tilde{B}))
\end{equation}
is unitary and satisfies
\[ \ad(u)\circ\iota_1= (\phi \otimes id_\Lambda)\circ \tau^{-1}\circ \jmath_1: \Lambda \to {}_1M_2\Lambda.\]
So by Lemma \ref{lem:conju1}, the following diagram commutes in $kk^h$
\begin{equation}\label{comm-fold}
\begin{tikzcd}
  \Lambda \ar[r,"\jmath_1"]
          \ar[dr,"\iota_1"', "\sim" {rotate=-30}]
 &\Lambda\oplus\Lambda \ar[d,"(\phi\otimes id_\Lambda)\circ\tau^{-1}"] \\
 &{}_1M_2\Lambda.
\end{tikzcd}
\end{equation}
Similarly, the diagram 
\begin{equation}\label{comm-fold2}
\begin{tikzcd}
  \Lambda \ar[r,"\jmath_2"]
          \ar[dr,"\iota_2"', "\sim" {rotate=-30}]
 &\Lambda\oplus\Lambda \ar[d,"(\phi\otimes id_\Lambda)\circ\tau^{-1}"] \\
 &{}_1M_2\Lambda
\end{tikzcd}
\end{equation}
commutes in $kk^h$. Let $\pi_i$ $(i=1,2)$ be as in Remark \ref{rem:sumtria}. Because diagrams \eqref{comm-fold} and \eqref{comm-fold2} commute and because by 
Lemma \ref{lem:hermstab2} we have $j^h(\iota_1)=j^h(\iota_2)$, the following solid arrow diagram commutes in $kk^h$
\[
\begin{tikzcd}[column sep = 2cm]
    U\Lambda \ar[r]
 &\Lambda^2 \ar[r, "\phi\otimes \id_\Lambda"]
 &  {}_1M_2\Lambda  \\
    \Lambda \ar[r, "\jmath_1-\jmath_2"] \ar[u, dashed]
 &  \Lambda\oplus\Lambda \ar[r, "\pi_1 + \pi_2"] 
                           \ar[u, "\tau^{-1}", "\sim"' {anchor=north, rotate=90}]
 &  \Lambda. \ar[u, "\iota_1"]
\end{tikzcd}
\]
Because the middle and right vertical arrows are isomorphisms in $kk^h$, we get that the dashed map is an isomorphism in $kk^h$. 
Next we prove the second isomorphism of the Lemma. Let 
\[ V\Lambda \to \Lambda \xrightarrow{\eta \otimes \id_\Lambda} \Lambda^2 \]
be the triangle in $kk^h$ defining $V\Lambda$. Let $t$ be as in \eqref{eq:twist}; one checks that the following square commutes
\begin{equation}\label{diag:ttau}
\begin{tikzcd}
    \Lambda \ar[r,"\eta \otimes \id_\Lambda"] \ar[d,equal] &\Lambda^2 \ar[d,"\tau"] \\
    \Lambda \ar[r,"\jmath_1+\jmath_2 t"] &\Lambda \oplus \Lambda.
\end{tikzcd}
\end{equation}
The map $\jmath_1+\jmath_2 t$ completes to a split distinguished triangle in $kk^h$
\[ \Lambda \xrightarrow{\jmath_1+ \jmath_2t} \Lambda \oplus \Lambda \xrightarrow{\pi_1 - t\pi_2}  \Lambda. \]
Rotating the split triangle above we get the triangle
\[ \Omega\Lambda \xrightarrow{0} \Lambda \xrightarrow{\jmath_1+\jmath_2 t} \Lambda \oplus\Lambda. \]
Next \eqref{diag:ttau} extends to a map of triangles in $kk^h$
\[
\begin{tikzcd}
    V\Lambda \ar[r] \ar[d, dashed]
 &\Lambda \ar[r, "\eta\otimes \id_\Lambda"] \ar[d,equal]
 &  \Lambda^2  \ar[d,"\tau"]\\
    \Omega\Lambda \ar[r,"0"]
 &  \Lambda \ar[r, "\jmath_1+\jmath_2 	t"]
 &  \Lambda.
\end{tikzcd}
\]
It follows that the dashed map is an isomorphism.
\end{proof}

\begin{lem}\label{lem:sigmavu-l}
There is an isomorphism
\[ j^h(\Sigma VU) \cong j^h(\ell).\]
In particular, $j^h(VU) \cong j^h(\Omega)$.
\end{lem}
\begin{proof}
As before, we omit $j^h$ from the notation. In view of Lemma \ref{lem:hermstab2}
it suffices to show that $\Sigma VU$ is $kk^h$-isomorphic to ${}_1M_2$.  Let
\[ VU \to U \xrightarrow{\eta \otimes id_U} \Lambda U \]
be the triangle in $kk^h$ that defines $VU$. The $kk^h$ isomorphism between $\Lambda U = U\Lambda$ and $\Lambda$ established in  Lemma \ref{lem:ul-vlo} is induced by mapping $\Lambda^2$ to $\Lambda \oplus \Lambda$ and then retracting onto the first coordinate. Using this fact we get that there is a map of triangles in $kk^h$
\[
\begin{tikzcd}
   U \ar[r, "\eta \otimes id_U"] \ar[d,equal] 
  &\Lambda U \ar[r] \ar[d,"\sim" {anchor=south, rotate=90}] 
  &\Sigma VU \ar[d,dashed] \\
  U \ar[r]
  &\Lambda \ar[r, "\phi"]
  & {}_1M_2.
\end{tikzcd}
\]
It follows that the dashed $kk^h$-map is an isomorphism. 

\end{proof}

\begin{rem}\label{rem:tenso}
By Example \ref{ex:tenso}, the isomorphisms of Lemmas \ref{lem:ul-vlo} and \ref{lem:sigmavu-l} induce isomorphisms $j^h(U\Lambda A)\cong j^h(\Lambda A)$, $j^h(V\Lambda A)\cong j^h(\Omega\Lambda A)$ and $j^h(VUA)=j^h(A)[1]$ for every $A\in\ahas$. 
\end{rem}

\section{Karoubi's Fundamental Theorem}\label{sec:fund}
\comment{In \cite{karfund} Karoubi defines similar functors $U$ and $V$. The main difference is that between his and ours is that Karoubi's versions of such functors are defined by $U_{Kar} := \Gamma_\phi$ and $V_{Kar} := \Gamma_\eta$, that is, the pullbacks of $\Sigma \phi$ and $\Sigma\eta$ like in example \ref{ex:cones}.}

\begin{thm}[Karoubi, \cite{karfund}]\label{thm:fund}
Assume that $\ell$ satisfies the $\lambda$-assumption \ref{stan:lambda}. There is an element $\theta_0 \in {}_{-1}K_2^h(U'^2)$ such that 

\item[i)] The composite ${}_{-1}K_2^h((U')^2)\to {}_{-1}K_2^h(\Sigma \inv(U'))\cong {}_{-1}K_1^h(\inv(U'))\cong K_1(U')\cong K_0(\Z)=\Z$ maps $\theta_0$ to $1$.
\item[ii)] For every unital $*$-$\ell$-algebra $R$, the product with $\theta_0$ induces an isomorphism
\[
\theta_0\star-:K_*^h(R) \iso K^h_{*+2}({}_{-1}M_2(U')^2R)
\]
\end{thm}
\begin{proof} By Lemma \ref{lem:kepsilon}, for any unital $*$-ring $S$, we have ${}_{-1}K^h_*(S)=K^h_*({}_{-1}M_2S)$. The element $\theta_0$ of the present theorem appears under the name of $\sigma$ in the first line of \cite{karfund}*{Section 3.1}. Thus as mentioned by Karoubi in \cite{karfund}*{line -3 of page 263 through line 3 of page 264}, the current theorem is just another way of phrasing Karoubi's fundamental theorem. Indeed, as said in Remark \ref{rem:uukar}, Karoubi showed that $\Omega K^h(V'U'R)\weq K^h(R)$,
so the theorem as stated here is equivalent to that proved in \cite{karfund}*{Section 3.5}, which says that cup-product with $\theta_0$ induces an isomorphism $K^h_*(V'R)\iso {}_{-1}K^h_{*+1}(U'R)$.
\end{proof}

\begin{rem}\label{rem:clau}
Max Karoubi pointed out to us that his proof of the fundamental Theorem in \cite{karfund} implicitly uses Clauwens' calculation in \cite{clau} that the ring $\ell=\Z[x]$ equipped with the involution $x^*=1-x$ has Witt group $W(\ell)=\Z$. 
\end{rem}

We shall use Theorem \ref{thm:fund} to prove the following.

\begin{thm}\label{thm:bivafund}
Assume that $\ell$ satisfies the $\lambda$-assumption \ref{stan:lambda}. Then the image $\theta$ of  $c_2(\theta_0)$ under the isomophism $KH_2^h({}_{-1}M_2(U'\ell)^2)\cong KH_0^h({}_{-1}M_2U^2)=kk^h(\ell,{}_{-1}M_2U^2)$ induces a natural isomorphism
\[
\theta_A:=\theta\otimes\id_{j^h(A)}:j^h(A)\iso j^h({}_{-1}M_2U^2A)\,\,\forall A\in\ahas.
\]
\end{thm}

\begin{coro}\label{coro:bivafund} Let $\epsilon\in\ell$ be unitary. For every $A\in \ahas$, 
$$j^h({}_{\epsilon}M_2VA)\cong j^h({}_{-\epsilon}M_2UA)[1].$$ 
\end{coro}
\begin{proof} It is immediate from Theorem \ref{thm:bivafund}, Lemma \ref{lem:sigmavu-l} and Remark \ref{rem:tenso} that $j^h(VA)\cong j^h({}_{-1}M_2UA)[1]$. The corollary follows from this applied to ${}_\epsilon M_2A$ using the isomorphism 
\begin{equation}\label{-epsilon}
{}_{-1}M_2({}_\epsilon M_2)\cong M_{\pm}({}_{-\epsilon}M_2)
\end{equation}
 and hermitian stability. 
\end{proof}

\begin{proof}[Proof of Theorem \ref{thm:bivafund}] By Example \ref{ex:tenso}
By Lemma \ref{lem:prodagree}, under the isomorphism of Proposition \ref{prop:agree}, $(\theta_A)_*$ corresponds to the cup-product with $\theta$. By Lemma \ref{lem:prodkhh} the latter is induced by the product of Corollary \ref{coro:prod-sigmaomega} i) with $\theta_0$, up to a degree shift. By Theorem \ref{thm:fund}, $\theta_0\star-:K^h_m(A)\to K^h_{m+2}({}_{-1}M_2(U')^2A)$ is an isomorphism for all $m\in\Z$ and all unital $A\in\ahas$. Using excision, we obtain that the same is true also for not necessarily unital $A$ if $m\le -2$. It follows that $\theta\star-:KH_*(A)\to KH_*({}_{-1}M_2U^2A)$ is an isomorphism for all $A$, concluding the proof. 
\end{proof}

We finish the section with a lemma that will be used in Section \ref{sec:12}.

\begin{lem}\label{lem:commutes}
Consider the isomorphisms $j^h(U\Lambda)\cong j^h(\Lambda)$ of Lemma \ref{lem:ul-vlo} and $M_2\Lambda\cong {}_{-1}M_2\Lambda$ of \eqref{map:lambda-matrix}. Then the following diagram commutes
\[
\xymatrix{j^h(\Lambda)\ar[r]^(.4){\theta_{\Lambda}}\ar[d]_{\wr}^{\iota}& j^h({}_{-1}M_2U^2\Lambda)\ar[d]^{\wr}\\
           j^h(M_2\Lambda)\ar[r]_(.45){\cong}&j^h({}_{-1}M_2\Lambda)}
\]
\end{lem}
\begin{proof} By part i) of Theorem \ref{thm:fund}, we have a commutative diagram in $kk^h$, where as usual we have omitted $j^h$, 
\begin{equation}\label{diag:eldeka}
\xymatrix{\ell\ar[d]^{\iota\eta}\ar[r]^(.4)\theta &{}_{-1}M_2U^2\ar[r]& {}_{-1}M_2\Lambda U\ar[d]^{\wr}\\
                                                   M_2\Lambda \ar[rr]^{\cong}  &&{}_{-1}M_2\Lambda.}
\end{equation}
Let $p=pr_1\circ\tau:\Lambda^2\to\Lambda$; we have 
\begin{equation}\label{map:elp}
p((x_1,x_2)\otimes (x_3,x_4))=(x_1x_3,x_2x_4).
\end{equation}
Tensoring \eqref{diag:eldeka} with $\Lambda$ and composing the resulting vertical maps with those induced by $p$, we get another commutative diagram
\begin{equation}\label{diag:prupi}
\xymatrix{\Lambda\ar[d]^{\iota}\ar[r]^(.41){\theta_\Lambda} &{}_{-1}M_2U^2\Lambda\ar[r]& {}_{-1}M_2\Lambda U\Lambda\ar[d]^{\wr}\\
                                                   M_2\Lambda \ar[rr]^{\cong}  &&{}_{-1}M_2\Lambda.}
\end{equation}
Using \eqref{map:elp} and the fact that the $kk^h$ isomorphism $U\Lambda\cong\Lambda$ is induced by first mapping to $\Lambda^2$ and then applying $p$, we obtain that the following diagram commutes in $kk^h$
\begin{equation}\label{diag:upi}
\xymatrix{U^2\Lambda\ar[r]\ar[d]^{\wr}&\Lambda U\Lambda\ar[dl]^{\wr}\\
           \Lambda}
\end{equation}
Appling ${}_{-1}M_2$ to \eqref{diag:upi} we obtain that the composite ${}_{-1}M_2U^2\Lambda\to {}_{-1}M_2\Lambda$ in diagram \eqref{diag:prupi}
is the map in the diagram of the proposition, finishing the proof. 
\end{proof}

\section{The bivariant 12 term exact sequence}\label{sec:12}

\begin{defi}\label{defi:pre12}
Let $A,B \in \ahas$, $\epsilon\in\ell$ a unitary, ${}_\epsilon kk^h(A,B)$ as in \eqref{defi:kk} and $t$ as in \eqref{eq:twist}. Let $\eta:\ell\to \Lambda$ and $\phi:\Lambda\to {}_1M_2$ be as in \eqref{eta} and \eqref{phi}. Put $\bar{\phi}=j^h(\iota_1)^{-1}\circ j^h(\phi)$. Set
\begin{align*}
    {}_{\epsilon}W(A,B)  &= \coker( {}_{\epsilon}kk^h(A,\Lambda B) \xrightarrow{\bar{\phi}_*} {}_{\epsilon}kk^h(A,B)) \\
    {}_{\epsilon}W'(A,B) &= \ker  ({}_{\epsilon}kk^h(A,B)         \xrightarrow{\eta_*} {}_{\epsilon}kk^h(A,\Lambda B)) \\
    k(A,B)  &= \{ x \in kk^h(A,\Lambda B) : x = t_*x\}/\{x = y + t_*y\} \\
    k'(A,B)  &= \{ x \in kk^h(A,\Lambda B) : x = - t_*x\}/\{x = y - t_*y\}
\end{align*}
If $\epsilon = 1$ we omit it from the notation.
\end{defi}
\begin{thm}[cf. \cite{karfund}*{Th\'eor\`eme 4.3}]\label{thm:12}
There is an exact sequence
\[
\begin{tikzcd}[transform shape, nodes={scale=0.8}, column sep={0.4cm}]
   k(A,\Omega B) \ar[r] 
 & {}_{-1}W(A,\Omega^2B) \ar[r] 
 & W'(A,B) \ar[r]
 & k'(A,\Omega B) \ar[r]
 & {}_{-1}W'(A,\Omega B)  \ar[r]
 & {}_{-1}W(A,\Omega B)  \ar[d] \\
   W(A,\Omega B) \ar[u] 
 & W'(A,\Omega B) \ar[l]
 & k'(A,\Omega B) \ar[l]
 & {}_{-1}W'(A,B)  \ar[l] 
 & W(A,\Omega^2 B) \ar[l]
 & k(A,\Omega B) \ar[l]
\end{tikzcd}
\]
\end{thm}
\begin{proof}
As above, we omit $j^h$ in our notation. Write $\rho$ for the map obtained upon tensoring the canonical map $U\to\Lambda$ with $\Omega ({}_{-1}M_2)$. Consider the following distinguished triangles in $kk^h$
\[
\begin{tikzcd}
     \Omega\Lambda \ar[r,"\partial"]& V \ar[r] \ar[d,"\theta", "\sim"' {anchor=south, rotate=90}]
     &\ell \ar[r, "\eta"]
     &\Lambda  \\
     \Omega^2 {}_{-1}M_2 \ar[r,"\delta"]& \Omega {}_{-1}M_2U \ar[r,"\rho"]
     &\Omega {}_{-1}M_2\Lambda  \ar[r, "\phi_{\Omega}"]
     &\Omega {}_{-1}M_2.
\end{tikzcd}
\]
Let $\tau:\Lambda^2 \iso \Lambda \oplus \Lambda$ be as in \eqref{map:tau}.
\[
\begin{tikzcd}[column sep = 1.25cm]
 \Omega\Lambda \ar[r,"\partial"] \ar[d,"{\Omega\Lambda}\eta"] 
&V            \ar[r, "\theta", "\sim"'] \ar[d,"V  \eta"]
&\Omega {}_{-1}M_2U \ar[r,"\rho"] \ar[d, "{\Omega {}_{-1} M_2 U }\eta"]
&\Omega {}_{-1}M_2\Lambda  \ar[d, "{\Omega ({}_{-1} M_2\Lambda)} \eta"] \\
 \Omega\Lambda^2  \ar[r,"\partial"] \ar[d, " {\Omega }\tau"]
&V\Lambda               \ar[r, "\theta", "\sim"'] \ar[d, "\sim" {anchor=south, rotate=90}]
&\Omega {}_{-1}M_2U\Lambda  \ar[r,"\rho"] \ar[d,"\sim" {anchor=south, rotate=90}]
&\Omega {}_{-1}M_2\Lambda^2  \ar[d,"\widetilde{\tau}"]\\
\Omega(\Lambda  \oplus \Lambda)  \ar[r, "\Omega(\pi_1 - t\pi_2)"]
 &\Omega \Lambda               \ar[r, equal] 
&\Omega \Lambda  \ar[r, "\Omega(\jmath_1-\jmath_2)"']
&\Omega (\Lambda  \oplus \Lambda) .
\end{tikzcd}
\]

A direct computation shows that $\tau\circ(\Lambda\otimes\eta):\Lambda\to\Lambda\oplus\Lambda$ is the diagonal map. Hence from the diagram we get following equality
\begin{equation}\label{eq:compo}
\widetilde{\tau}(\Omega{}_{-1}M_2\eta)\rho\theta\partial
= \Omega((\jmath_1-\jmath_2)) (\pi_1-t \pi_2)(\jmath_1+\jmath_2)).
\end{equation}

Similarly, for $h_{-1}$ as in Examples \ref{exas:hermele} and $\iota_1$ the upper left-hand corner inclusion, we have
\[\widetilde{\tau}({\Omega {}_{-1}M_2 }\eta) =
    (\Omega (\jmath_1+\jmath_2))(\iota_1)^{-1}\ad(1,h_{-1}^{-1}).\]
Therefore, composing both sides of the equality \eqref{eq:compo} on the left with the projection onto the first coordinate, we get
\[ (\iota_1)^{-1}\ad(1,h_{-1}^{-1}) \rho \theta\partial = \Omega((\pi_1 - t \pi_2)(\jmath_1+\jmath_2) = \id-t.\]
Because the $*$-algebras involved in the argument above are flat, for any $B\in\ahas$ we map apply the functor $\otimesl B$ of Example \ref{ex:tenso} to obtain the same identity between $kk^h$ maps $\Omega \Lambda B\to \Omega\Lambda B$. Next apply the functor $kk^h(A,-)$ and use the identification ${}_{-1}kk^h(A,\Lambda B) \cong kk^h(A,\Lambda B)$, to obtain that $\rho\theta\partial$ induces
$\id-t_*$ on $kk^h(A,\Omega \Lambda B)$. Using this together with \eqref{eq:1+t}, the rest of the proof proceeds exactly as in  \cite{karfund}*{Th\'eor\`eme 4.3}.
\end{proof}

\begin{rem}\label{rem:12}
Theorem \ref{intro:12} follows from Theorem \ref{thm:12} applied with $B$ replaced by ${}_\epsilon M_2\Omega^n\Sigma B$ if $n\ge 0$ and by ${}_\epsilon M_2\Sigma^{-n-1}B$ otherwise, using the isomorphism \eqref{-epsilon} and hermitian stability. 
\end{rem}
\begin{rem}\label{rem:13} 
Let $\fC$ and $H:\ahas\to\fC$ be as in Proposition \ref{prop:kkuniv2}. The same argument as in Theorem \ref{thm:12} proves an analogous exact sequence for the groups obtained substituting $H(-)$ for $kk^h(A,-)$ in Definition \ref{defi:pre12}.
\end{rem}


\begin{bibdiv}
\begin{biblist}
\comment{\bib{lpabook}{book}{
author={Abrams, Gene},
author={Ara, Pere},
author={Siles Molina, Mercedes},
title={Leavitt path algebras}, 
date={2017},
series={Lecture Notes in Math.},
volume={2008},
publisher={Springer},
doi={$10.1007/978-1-4471-7344-1$},
}
}
\bib{bass}{article}{
author={Bass, Hyman},
title={Unitary algebraic $K$-theory},
book={
title={Algebraic $K$-theory III},
editor={Bass, Hyman},
series={Lecture Notes in Math.},
volume={343},
publisher={Springer},
},
pages={57--265},
}
\bib{battikh}{article}{
author={Battikh, Naoufel},
title={Some applications of the fundamental theorem for hermitian $K$-theory},
journal={Missouri J. Math. Sci.},
volume={23},
number={1},
date={2011},
pages={48-64},
}
\comment{
\bib{black}{book}{
AUTHOR = {Blackadar, Bruce},
     TITLE = {{$K$}-theory for operator algebras},
    SERIES = {Mathematical Sciences Research Institute Publications},
    VOLUME = {5},
 PUBLISHER = {Springer-Verlag, New York},
      YEAR = {1986},
     PAGES = {viii+338},
      ISBN = {0-387-96391-X},
  review = {\MR{859867}},
       URL = {http://dx.doi.org/10.1007/978-1-4613-9572-0},
}}


\bib{clau}{article}{
author={Clauwens, Frans J. B. J.},
title={The $K$-theory of almost symmetric forms},
booktitle={Topological structures, II (Proc. Sympos. Topology and Geom., Amsterdam, 1978), Part 1},
pages={41–-49}, 
series={Math. Centre Tracts},
volume={115}, 
publisher={Math. Centrum}, 
address={Amsterdam}, 
date={1979},
}
\bib{friendly}{article}{
   author={Corti\~nas, Guillermo},
   title={Algebraic v. topological $K$-theory: a friendly match},
   conference={
      title={Topics in algebraic and topological $K$-theory},
   },
   book={
	    editor={Corti\~nas, Guillermo},
      series={Lecture Notes in Math.},
      volume={2008},
      publisher={Springer, Berlin},
   },
   date={2011},
   pages={103--165},
   review={\MR{2762555}},
}
\bib{corel}{article}{
author={Corti\~nas, Guillermo},
author={Ellis, Eugenia},
title={Isomorphism conjectures with proper coefficients},
journal={J. Pure Appl. Algebra},
volume={218},
year={2014},
pages={1224--1263},
}
\comment{
\bib{jot}{article}{
   author={Corti\~nas, Guillermo},
   author={Rodr\'\i guez, Mar\'\i a Eugenia},
   title={$L^p$-operator algebras associated with oriented graphs},
   journal={J. Operator Theory},
   volume={81},
   date={2019},
   pages={225--254},
   }}
\bib{ct}{article}{
    AUTHOR = {Corti\~nas, Guillermo},
    author={Thom, Andreas},
     TITLE = {Bivariant algebraic {$K$}-theory},
  JOURNAL = {J. Reine Angew. Math.},
    VOLUME = {610},
      YEAR = {2007},
     PAGES = {71--123},
      ISSN = {0075-4102},
       DOI = {10.1515/CRELLE.2007.068},
       URL = {https://doi.org/10.1515/CRELLE.2007.068},
}
\bib{newlook}{article}{
author={Cuntz, Joachim},
title={A new look at $KK$-theory},
journal={$K$-theory},
volume={1},
date={1987},
pages={31--51},
}

\bib{cmr}{book}{
 AUTHOR ={Cuntz, Joachim},
 author={Meyer, Ralf},
 author={Rosenberg, Jonathan M.},
     TITLE = {Topological and bivariant {$K$}-theory},
    SERIES = {Oberwolfach Seminars},
    VOLUME = {36},
 PUBLISHER = {Birkh\"{a}user Verlag, Basel},
      YEAR = {2007},
     PAGES = {xii+262},
      ISBN = {978-3-7643-8398-5},
      }
\bib{GJ}{book}{
author={Goerss, Paul G.},
author={Jardine, John F.},
title={Simplicial homotopy theory},
publisher={Birkh\"auser},
year={1999},
}

			\comment{
\bib{kkg}{article}{
    AUTHOR = {Ellis, Eugenia},
     TITLE = {Equivariant algebraic {$kk$}-theory and adjointness theorems},
   JOURNAL = {J. Algebra},
    VOLUME = {398},
      YEAR = {2014},
     PAGES = {200--226},
      ISSN = {0021-8693},
       DOI = {10.1016/j.jalgebra.2013.09.023},
       URL = {https://doi.org/10.1016/j.jalgebra.2013.09.023},
}
}
\bib{karfund}{article}{
    AUTHOR = {Karoubi, Max},
    title={Le th\'eor\`eme fondamental de la $K$-th\'eorie hermitienne},
    journal={Ann. Math.},
    volume={112},
    year={1980},
    pages={259--282},
}
\bib{kar343}{incollection}{
author={Karoubi, Max},
title={Periodicit\'e de la $K$-th\'eorie hermitienne},
book={
editor={Bass, Hyman},
title={Algebraic $K$-theory III},
series={Lecture Notes in Math.},
publisher={Springer},
volume={343},
},
pages={301--411},
}
\bib{karloca}{article}{
author={Karoubi, Max},
title={Localisation des formes quadratiques II},
journal={Ann. Sci. \'Ec. Norm. Sup\'er. (4)},
volume={8},
number={1},
year={1975},
pages={99--155},
}
\bib{kv2}{article}{
    AUTHOR = {Karoubi, Max},
    author={Villamayor, Orlando},
     TITLE = {{$K$}-th\'{e}orie alg\'{e}brique et {$K$}-th\'{e}orie topologique. {II}},
   JOURNAL = {Math. Scand.},
    VOLUME = {32},
      YEAR = {1973},
     PAGES = {57--86},
      ISSN = {0025-5521},
       DOI = {10.7146/math.scand.a-11446},
       URL = {https://doi.org/10.7146/math.scand.a-11446},
       }
\bib{loday}{article}{
author={Loday, Jean Louis},
title={$K$-th\'eorie alg\'ebrique et repr\'esentations de groupes},
journal={Ann.Sci.\'Ec. Norm. Sup.},
volume={9},
date={1976},
pages={309--377},
}
\comment{
\bib{ojapa}{article}{
    AUTHOR = {Ojanguren, Manuel},
    author={Panin, Ivan},
     TITLE = {The {W}itt group of {L}aurent polynomials},
   JOURNAL = {Enseign. Math. (2)},
    VOLUME = {46},
      YEAR = {2000},
    NUMBER = {3-4},
     PAGES = {361--383},
      ISSN = {0013-8584},
}
\bib{vorst}{article}{
author={Vorst, Ton},
title={Localization of the $K$-theory of polynomial extensions},
journal={Math. Annalen},
volume={244},
date={1979},
pages={33--54},
}
}
\bib{mvsnk}{incollection}{
author={Weibel, Charles A.},
title={Meyer-Vietoris sequences and module structures on $NK_*$},
book={
title={Algebraic $K$-theory},
editor={Stein, Michael R.},
editor={Friedlander, Eric M.},
series={Lecture Notes in Math.},
publisher={Springer},
volume={854},
year={1981},
},
pages={466--493},
}
\bib{kh}{article}{
   author={Weibel, Charles A.},
   title={Homotopy algebraic $K$-theory},
   conference={
      title={Algebraic $K$-theory and algebraic number theory (Honolulu, HI,
      1987)},
   },
   book={
      series={Contemp. Math.},
      volume={83},
      publisher={Amer. Math. Soc.},
      place={Providence, RI},
   },
   date={1989},
   pages={461--488},
}

\end{biblist}
\end{bibdiv}
\end{document}